\documentclass[11pt]{article}

\usepackage{amsmath,amssymb,array,latexsym,amsthm,pdfsync,hyperref}
\theoremstyle{definition}

\newtheorem{thm}{Theorem}
\newtheorem{lem}{Lemma}
\newtheorem{cor}{Corollary}
\newtheorem{prop}{Proposition}

\theoremstyle{remark}
\theoremstyle{remark}\newtheorem{rem}{Remark}
\theoremstyle{remark}\newtheorem{notat}{Notation}

\def\lra{\longrightarrow}

\def\n{\noindent}

\newcommand{\E}{\ensuremath{\mathbb E}}

\begin{document}

\title{Empirical Quantile CLTs for Time Dependent Data}

\date{ }
\author{James Kuelbs \\ Department of Mathematics\\   University of Wisconsin  \\ Madison, WI 53706-1388 \ \and Joel Zinn \\ Department of Mathematics \\ Texas A\&M University \\ College Station, TX 77843-3368 }

\maketitle

\begin{abstract}  We establish empirical quantile process CLTs based on $n$ independent copies of a stochastic process $\{X_t: t  \in E\}$ that are uniform in $t \in E$ and quantile levels $\alpha \in I$, where $I$ is a closed sub-interval of $(0,1)$. Typically $E=[0,T]$, or a finite product of such intervals. Also included are CLT's  for the empirical process based on $\{I_{X_t \le y} - \rm {Pr}(X_t \le y): t \in E, y \in R \}$ that are uniform in $t \in E, y \in R$. The process $\{X_t: t \in E\}$ may be chosen from a broad collection of Gaussian processes, compound Poisson processes, stationary independent increment stable processes, and martingales.
\end{abstract}

\section{Introduction}\label{sec1}

\indent 

Let  $X=\{X(t)\colon\ t \in E\}$ be a stochastic process with $P(X(\cdot) \in D(E))=1,$ where $E$ is a set and $D(E)$ is a collection of real valued 
functions on $E$. Also, let $\mathcal{C}=\{C_{s,x}\colon\  s \in E, x \in \mathbb{R}\},$ where $C_{s,x}=\{z \in D(E)\colon\  z(s) \le x\}, s \in E, x \in \mathbb{R}$. If  $\{X_{j}\}_{j=1}^{\infty}$ are i.i.d. copies of the stochastic process $X$ and $F(t,x):=P(X(t) \le x)=P( X(\cdot) \in C_{t,x})$, then the empirical distributions built on $\mathcal{C}$ (or built on the process $X$) are defined by 
\[
F_n(t,x)= \frac{1}{n} \sum_{i=1}^n I_{(-\infty, x]}(X_i(t))= \frac{1}{n}\sum_{i=1}^n I_{\{X_i \in C_{t,x}\}}, C_{t,x} \in \mathcal{C},
\]
and we say $X$ is the input process.

The empirical processes indexed by $\mathcal{C}$ (or just $E\times\mathbb{R})$ and built from the process, $X$, are given by
\[\nu_{n}(t,x):=\sqrt{n}\bigl(F_{n}(t,x)-F(t,x)\bigr).
\] 

In \cite{kkz} we studied the central limit theorem in this setting, that is, we found sufficient conditions for a pair $(\mathcal{C}, P)$, where $P$ is the law of $X$ on $D(E)$, ensuring that the sequence of empirical processes $\{\nu_n(t,x): (t,x) \in E\times \mathbb{R}\}, n \ge1,$
converge to a centered Gaussian process, $G=\{G_{t,x}\colon\  (t,x)\in E\times\mathbb{R}\}$ with covariance
$$
\E(G(s,x)G(t,y)) = \E([I(X_s \leq x)-P(X_s \leq x)][I(X_t \leq y) -P(X_t \leq y)])
$$
This requires that the law of $G$ on $\ell_{\infty}(E \times \mathbb{R})$ (with the usual sup-norm) be Radon, or equivalently, (see Example 1.5.10 in \cite{vw}), that $G$ has sample paths which are bounded and uniformly continuous on $E \times \mathbb{R}$ with respect to the psuedo-metric
\begin{equation}\label{eq1.1}
d((s,x),(t,y)) 
= \E([I(X_s \leq x)-I(X_t \leq y) -(P(X_s \leq x) - P(X_t \leq y))]^2)^{\frac{1}{2}}.
\end{equation}
It also requires that for every bounded, continuous $F\colon\  \ell_{\infty}(E\times \mathbb{R})\lra \mathbb{R}$, 
\[
\lim_{n\to\infty}\E^{*}F(\nu_{n})=\E F(G),
\]
where $\E^{*}$ denotes the upper expectation (see, e.g., p. 94 in \cite{Dudley-unif-clt}).

The quantiles and empirical quantiles are defined as the left-continuous inverses of $F(t,x)$ and $F_{n}(t,x)$ in the variable $x$, respectively: 
\begin{equation}\label{quantile} \tau_{\alpha}(t)=F^{-1}(t,\alpha)=\inf\{x\colon F(t,x)\ge \alpha\}
\end{equation}
and 
\begin{equation}\label{empirical quantile} \tau^{n}_{\alpha}(t)=F_{n}^{-1}(t,\alpha)=\inf\{x\colon F_{n}(t,x)\ge \alpha\}.
\end{equation}

The empirical quantile processes are defined as 
\begin{equation}\sqrt{n}\bigl(F_{n}^{-1}(t,\alpha)-F^{-1}(t,\alpha)\bigr)\notag,
\end{equation}
and we also use the more compact notation
\begin{equation*}\sqrt{n}\bigl( \tau^{n}_{\alpha}(t)- \tau_{\alpha}(t)\bigr),
\end{equation*}
for these processes. Since we are seeking limit theorems with non-degenerate Gaussian limits, it is appropriate to mention that for $\alpha \in (0,1)$ and $t$ fixed, that is, for a one-dimensional situation, a necessary condition for the weak convergence of 
\begin{equation}\label{alpha t fixed}\sqrt{n}\bigl(\tau_{\alpha}^{n}(t)-\tau_{\alpha}(t)\bigr)\Longrightarrow \xi,
\end{equation}
where $\xi$ has a strictly increasing, continuous distribution, is that the distribution function $F(t,\cdot)$ be differentiable at $\tau_{\alpha}(t)$ and $F'(t, \tau_{\alpha}(t))>0$. Hence $F(t,\cdot)$ is strictly increasing near $\tau_{\alpha}(t)$ as a function of $x$, and if we keep $t$ fixed, but ask that (\ref{alpha t fixed}) holds for all $\alpha \in (0,1)$, then $F(t,x)$ will be differentiable, with strictly positive derivative $F'(t,x)$ on the the set $J_t=\{x: 0< F(t,x)<1\}$. Moreover, by Theorem 8.21, p 168, of \cite{rudin}, if $F'(t,x)$ is locally in $L_1$ with respect to Lebesgue measure on $J_t$, then $F'(t,x)$ is the density of $F(t,\cdot)$ and it is strictly positive  on $J_t$. For many of the base processes we study here, $J_t = \mathbb{R}$ for all $t \in E$, but should that not be the case, it can always be arranged by adding an independent random variable $Z$ with strictly positive density to our base process in order to have a suitable input  process.  In p
 articular, the reader should consider a base process as one which, after possibly some modification, will be a suitable input process. At first glance perhaps this may seem like a convenient shortcut, but we know from \cite{kkz} that when $E=[0,T]$ and the base process is a fractional Brownian motion starting at zero when $t=0$, then the empirical CLT over $\mathcal{C}$ fails, but by adding $Z$ as indicated above it will hold.  In these cases adding $Z$ is just starting the process with $Z$. Hence a typical assumption throughout sections 2, 3, and 4 will be that the distributions $F(t,\cdot)$ are continuous and strictly increasing on $\mathbb{R}$, but in section 5 we show how to remove this assumption for fractional Brownian motions and symmetric stable processes with stationary independent increments  when $E=[0,T]$, and the processes start at zero at $t=0$.

In section \ref{vervaat} we extend a result of Vervaat  \cite{vervaat-quantiles} on the relation between empirical and quantile processes. Because of this relationship, when such a theorem is applied to empirical quantile processes built on a stochastic process, of necessity, the hypotheses includes the CLT for the empirical process as well as the existence of and conditions on the densities of $F(t,x)$. Hence, in order to prove empirical quantile CLT's for stochastic processes we will need the associated empirical CLT's. Some of these CLT's follow from the results in \cite{kkz}, but in section \ref{applications} we also show that several other classes of processes fall within the scope of those results. This includes the cases when the base process is a strictly stable process with stationary and independent increments, certain martingales, and even other independent increment processes. Section 4 then turns to the task of obtaining the empirical quantile process CLT's for t
 hese examples, and as mentioned above, section 5 looks at empirical quantile results for some important examples where one can get around the difficulty imposed when the input process starts zero at time zero. 
The results of section 5 were motivated by the CLT for the sample median of independent Brownian motions with value $0$ at $0$, a result of J. Swanson (\cite{swanson-scaled-median}), and we extend that result by proving a CLT for such empirical quantiles uniform, not only in the time parameter, but also uniform in all quantiles. Moreover, we do this for symmetric independent increment stable processes and fractional Brownian motions.

\section{Vervaat's Approach}\label{vervaat}

Throughout we assume the notation of section one. In particular, in this section we are assuming that for all $t \in E$, $F(t,x)$ is strictly increasing and continuous in $x \in \mathbb{R}$. Our goal in this section is to prove an analogue of Vervaat's Lemma 1 in \cite{vervaat-quantiles}. We follow Vervaat's idea of using an almost sure version of the empirical CLT.  
 
\begin{notat}For a function $f: S\lra \bar{\mathbb{R}}$ we use the notation $f^{*}$ to denote a measurable cover function (see Lemma 1.2.1 \cite{vw}). 
\end{notat}

An important result regarding weak convergence or convergence in law, in a general context, is that it also has a form allowing almost sure convergence. Such results have a long history, and here we use Theorem 3.5.1 in \cite{Dudley-unif-clt}, which is slightly more general.   

\begin{thm}[\cite{Dudley-unif-clt}]\label{thm2.1}
Let $(D,d_{\infty})$ be a metric space, $({\Omega}, \mathcal{A}, Q)$ be a probability space and $f_n\colon\  {\Omega} \rightarrow D$ for each $n=0,1,\cdots$. Suppose $f_0$ has separable range, $D_0$, and is measurable with respect to the Borel sigma algebra on $D_0$. Then $\{f_n\colon\  n\geq1\}$ converges weakly, or in law,  to $f_0$  iff there exists a probability space $(\widehat{\Omega}, \widehat{\mathcal{F}},\widehat{P})$ and perfect measurable functions $g_n$ from $(\widehat{\Omega},\widehat{\mathcal{F}})$ to $({\Omega}, \mathcal{A})$ for $n=0,1,\cdots$, such that
\begin{equation}
\label{distribution-perfect}
\widehat{P}\circ g_n^{-1} = Q \text{ on } \mathcal{A} 
\end{equation}
for each $n$, and
\begin{equation}\label{a.s. perfect}
d_{\infty}^{*}(f_n\circ g_n, f_0\circ g_0) \underset{\text{a.s}}{\rightarrow} 0. 
\end{equation}
where $d_{\infty}^{*}(f_n\circ g_n, f_0\circ g_0)$ denotes the measurable cover function for $d_{\infty}(f_n\circ g_n, f_0\circ g_0)$ and the $a.s.$ convergence is with respect to $\widehat P$.
\end{thm}

In our setting the metric space $D$ is $\ell_{\infty}(E \times \mathbb{R})$, with distance $d_{\infty}$ the usual sup-norm there, and the probability space $(\Omega, \mathcal{A}, Q)$  supports the i.i.d.\ sequence $\{X_j\colon\  j \geq 1\}$ and the Gaussian process $G$. 
Then, for $\omega \in \Omega, n \geq 1,$ the $f_{n}$ of Dudley's result is our $\nu_{n}$,
\begin{equation}\label{old1.4}
f_n(\omega) = \sqrt n(F_n(\cdot,\cdot)(\omega)-F(\cdot,\cdot)) \in \ell_{\infty}(E \times R), 
\end{equation}
and $\{f_n\colon\  n\geq 1\}$ converges in law to
\begin{equation}\label{old1.5}
f_0(\omega)= G(\cdot,\cdot)(\omega) \in \ell_{\infty}(E \times R).
\end{equation}
That is, we are assuming the empirical  CLT over $\mathcal{C}$, and therefore 
Theorem \ref{thm2.1}
implies there is a suitable probability space $(\widehat{\Omega},\widehat{\mathcal{F}},\widehat{P})$ and a set
$\widehat{\Omega}_1 \subset \widehat{\Omega}$ with $P(\widehat{\Omega}_1)=1$  such that for all $\hat{\omega}\in\widehat{\Omega}_{1}$, 
\begin{equation}\label{old1.6}
|| (f_n \circ g_n)(\hat{\omega}) - (f_0 \circ g_0)(\hat{\omega})||^{*} \equiv (\sup_{t \in E, x  \in \mathbb{R}} |(f_n \circ g_n)(\hat{\omega})-(f_0 \circ g_0)(\hat{\omega})|)^{*}\rightarrow 0. 
\end{equation}
Hence, if
$$
\widehat F_n(t,x)(\hat{ \omega})= F_n(t,x)(g_n(\hat{\omega})) \text{ and } \widehat G(t,x)(\hat{\omega})=G(t,x)(g_0(\hat{\omega})), 
$$
then on $\widehat{\Omega_1}$ we have the empirical distribution functions $\{\widehat F_n\colon\  n\geq 1\}$ satisfying
\begin{equation}\label{old1.7}
\|\sqrt n( \widehat F_n- F) - \widehat G \|^{*} 
\equiv (\sup_{t \in E, x  \in \mathbb{R}} |\sqrt n(\widehat F_n(t,x)(\hat{ \omega}) -F(t,x))-\widehat G(t,x)(\hat{\omega})|)^{*}\rightarrow 0.
\end{equation}

\begin{rem}\label{rem1.1}
The functions $\widehat F_n$ are still distribution functions as functions of $x$, and on $\Omega$ we have $\sqrt n(\widehat F_n -F)-\widehat G \in \ell_{\infty}(E \times \mathbb{R})$. In addition, since the functions $\{g_n\colon\  n\geq 0\}$ are perfect and (\ref{distribution-perfect}) holds, it follows for every bounded, real valued function $h$ on $\ell_{\infty}(E\times  \mathbb{R})$, and $n\geq 1$, that
\begin{equation}\label{old1.8}
\E_{\widehat{P}}^{*}[h(\sqrt n(\widehat F_n-F))] = \E_Q^{*}[h(\sqrt n( F_n-F))] \text{ and } \E_{\widehat{P}}[h(\widehat G)] = \E_Q[h(G)]. 
\end{equation}
Since we are assuming $\{\sqrt n( F_n-F)\colon\  n \geq 1\}$ converges weakly to the Gaussian limit $G$, and $G$ has separable support in $\ell_{\infty}(E\times \mathbb{R})$, then \eqref{old1.8} immediately implies $\{\sqrt n(\widehat F_n-F)\colon\  n \geq 1\}$ also converges weakly to $G$.
\end{rem}

The generalized inverse of $\widehat F_n(t,\cdot$) in the second variable is given by
\begin{equation}\label{tau_{n} hat}
\widehat \tau_{\alpha}^n(t)\equiv \widehat F_n^{-1}(t,\alpha)=\inf\{x\colon\  \widehat F_n(t,x) \ge \alpha\},~t\in E, \alpha \in (0,1), n \ge1,
\end{equation}
and as before for each $t \in E, \alpha \in (0,1)$, the inverse function
\begin{equation}\label{tau}
\tau_{\alpha}(t)\equiv  F^{-1}(t,\alpha)=\inf\{x: F(t,x) \ge \alpha\}. 
\end{equation}
Of course, since we are assuming $F(t,x)$ is strictly increasing, this is a classical inverse function, and to emphasize that the inverse is only on the second variable we also will write $\widehat F_{n,t}^{-1}$ and $F_{t}^{-1}$ for these inverses. Then, for each $t \in E$ we have  $\widehat F_{n,t}^{-1}(\cdot)\colon\ [0,1] \overset{\text{into}}{\longrightarrow} \mathbb{R}$ and since $F_t$ is assumed continuous we have $F_{t}^{-1}(\cdot)\colon\ (0,1) \overset{\text{onto}}{\longrightarrow} \mathbb{R}$. It is also useful to define $F_t^{-1}(0)=-\infty, F_{t}(-\infty)= \widehat F_{n,t}(-\infty)=0, F_t^{-1}(1)=\infty, F_{t}(+\infty)= \widehat F_{n,t}(+\infty)=1$, and $\widehat G_t(-\infty)= \widehat G_t(+\infty)=0$. We also set  $\overline{\mathbb{R}}= \mathbb{R} \cup \{-\infty \} \cup \{ +\infty \}$.

To use (\ref{old1.7})  we will need the function $\widehat F_{n,t} \circ F_t^{-1}$ and its inverse, which is determined in the next lemma.

\begin{lem}\label{lem2.1}
For each $t \in E$
\begin{equation}\label{eq2.3}
(\widehat F_{n,t}\circ F_t^{-1})^{-1}= F_t \circ \widehat F_{n,t}^{-1} 
\end{equation}
where the inverses are defined as in (\ref{tau_{n} hat}) and (\ref{tau}).
\end{lem}

\begin{proof}
 For each $t \in E, \alpha \in [0,1]$, we have, since we are assuming $F_t(\cdot)$ is strictly increasing and continuous, that
\begin{align*}
(\widehat F_{n,t} \circ F_t^{-1})^{-1}(\alpha) &= \inf \{ \beta\colon\  \widehat F_{n,t} \circ F_{t}^{-1}(\beta) \geq \alpha \}\\ 
&= \inf \{F_t(x)\colon\  \widehat F_{n,t}(x) \geq \alpha \}\\
& = F_t(\inf \{x\colon\  \widehat F_{n,t}(x) \geq \alpha\})\\
& = (F_t\circ \widehat F_{n,t}^{-1})(\alpha).\qquad\qed
\end{align*}
\renewcommand{\qed}{}\end{proof}

The next lemma is our modification of Lemma 1 in \cite{vervaat-quantiles} applicable to the present situation. 

\begin{lem}\label{lem2.2}
Let $a_{n} \rightarrow 0$, assume that uniformly in  $t \in E$, $\widehat G_t(F_t^{-1}(\alpha))$ is a uniformly continuous function of $\alpha \in (0,1)$, 
 and
\begin{equation} \label{old 2.4}
\left(\sup_{t \in E, x \in \mathbb{R}}\left|\frac{\widehat F_{n,t}(x) - F_t(x)}{a_n} -\widehat G_t(x)\right|\right)^{*} \rightarrow 0 
\end{equation}
as $n$ tends to infinity. Then, setting $I_t(\alpha)= \alpha$ for $t \in E, \alpha \in [0,1]$, we have

\begin{equation}\label{old 2.5}
\left(\sup_{t \in E, \alpha \in [0,1]}\left|\frac{(\widehat F_{n,t} \circ F_t^{-1})(\alpha) - I_t(\alpha)}{a_n} -(\widehat G_t \circ F_t^{-1})(\alpha)\right|\right)^{*} \rightarrow 0, 
\end{equation}
and
\begin{equation}\label{old 2.6}
\left(\sup_{t \in E, u \in [0,1]}\left|\frac{(F_{t} \circ \widehat F_{n,t}^{-1})(u) - I_t(u)}{a_n} +(\widehat G_t \circ F_t^{-1}(u)\right|\right)^{*} \rightarrow 0.
\end{equation}
\end{lem}

\begin{proof} 
Since we are assuming for each $t \in E$ that $F(t,\cdot)$ is strictly increasing and continuous on $\mathbb{R}$, it follows that $\{F_t^{-1}(\alpha)\colon\   \alpha \in (0,1)\}= \mathbb{R}$. Therefore, if one restricts $\alpha$ in (\ref{old 2.5}) to be in $(0,1)$, then (\ref{old 2.5}) follows immediately from (\ref{old 2.4}). To obtain (\ref{old 2.5}) for $\alpha=0$ and $\alpha=1$, then follows from the  conventions we made prior to the statement of the lemma involving $\pm \infty$.

To show (\ref{old 2.5}) implies (\ref{old 2.6}) we define for each $t \in E$ the completed graph of $(\widehat F_{n,t} \circ F_t^{-1})(\cdot)$  on $[0,1]$ to be given by 
$$
\Gamma_{n,t} = \{(\alpha,u)\colon\   \alpha \in [0,1], (\widehat F_{n,t}\circ F_t^{-1})(\alpha - 0) \leq u \leq  (\widehat F_{n,t} \circ F_t^{-1})(\alpha + 0) \}.
$$
Here $ (\widehat F_{n,t}\circ F_t^{-1})(\alpha \pm 0)$ denotes the left and right hand limits of $ (\widehat F_{n,t}\circ F_t^{-1})(\cdot)$ when $\alpha \in (0,1)$, and are given through the conventions above when $\alpha=0$ or 1, i.e.\ we understand the left and right hand limits at zero to both be zero, and the left and right hand limits at one to both be one.
Now (\ref{old 2.5}) implies that
\[
\lim_{n \rightarrow \infty} \left(\sup \left\{\left|\tfrac{ u - \alpha}{a_n} - (\widehat G_t \circ F_{t}^{-1})(\alpha)\right|\colon\  (t,\alpha) \in E \times [0,1], (\alpha,u) \in \Gamma_{n,t}\right\}\right)^{*}=0,
\]
which also implies that
\begin{equation} \label{old 2.7}
 \lim_{n \rightarrow \infty} \left(\sup \left\{\left|\tfrac{  \alpha - u}{a_n} + (\widehat G _t \circ F_{t}^{-1})(\alpha)\right|\colon\  (t,\alpha) \in E \times [0,1], (\alpha,u) \in \Gamma_{n,t}\right\}\right)^{*}=0. 
\end{equation}
For each $t \in E$ we set
\[
\Gamma_{n,t}^{-1} = \{(u,\alpha)\colon\    u \in [0,1], (F_{t} \circ \widehat F_{n,t}^{-1})(u-0) \leq \alpha \leq  (F_{t}\circ \widehat F_{n,t}^{-1})(u+0)\},
\]
where one can check that the left hand limit of $F_t \circ \widehat F_{n,t}^{-1}(\cdot)$ at zero is zero and we take the right hand limit at one to be one.

Then one can check that 
\begin{equation}\label{old 2.8}
(\alpha,u) \in \Gamma_{n,t} \text{ if and only if } (u,\alpha) \in \Gamma_{n,t}^{-1}.
\end{equation}
Moreover, (\ref{old 2.6}) is implied by
\begin{equation}\label{old 2.9}
 \lim_{n \rightarrow \infty} \left( \sup \left\{\left|\tfrac{  \alpha - u}{a_n} + (\widehat G _t \circ F_{t}^{-1})(u)\right|\colon\  (t,u) \in E \times [0,1], (u,\alpha) \in \Gamma_{n,t}^{-1}\right\}\right)^{*}=0,  
\end{equation}
and (\ref{old 2.7}) and (\ref{old 2.8}) implies (\ref{old 2.9}) provided we show
\begin{equation}\label{old 2.10}
 \lim_{n \rightarrow \infty} (\sup \{|(\widehat G _t \circ F_{t}^{-1})(u) - (\widehat G _t \circ F_{t}^{-1})(\alpha) |\colon\   t \in E, (u,\alpha) \in \Gamma_{n,t}^{-1}\})^{*}=0. 
\end{equation}

Since we are assuming the empirical CLT holds over $\mathcal{C}$ with Gaussian limit process $\{G(t,x): (t,x) \in E \times \mathbb{R}\}$, it follows that $G$ has a version which is sample uniformly continuous
on $E\times \mathbb{R}$ with respect to its $L_2$-distance, $d(\cdot,\cdot)$, given in (\ref{eq1.1}). This is a consequence of the addendum to Theorem 1.5.7 of \cite{vw}, p. 37.
When referring to $G$ we will mean this version. By the total boundedness of the distance, the associated space of uniformly continuous functions is separable in the uniform topology. This space is a closed subspace of  $\ell_{\infty}(E \times \mathbb{R})$, which then implies that this version of $G$ is measurable with respect to the Borel sets of $\ell_{\infty}(E \times \mathbb{R})$.
Using the definition of $\widehat{G}$ following (\ref{old1.6}), and (\ref{distribution-perfect}) with $n=0$, we
have the laws of $G$ and $\widehat{G}$ are equal on $\ell_{\infty}(E \times \mathbb{R})$. In particular, $\widehat{G}$ is also measurable to the Borel sets of $\ell_{\infty}(E\times \mathbb{R})$ with separable support there, it has the same covariance $G$ and its $L_2$-distance is $d$, and it is sample continuous on $(E \times \mathbb{R}, d) $ with $\widehat{P}$-probability one.

Now for each $t \in E$ and $\alpha, \beta \in (0,1)$ we have
\begin{equation}\label{old 2.11}
 d((t,F_t^{-1}(\alpha),(t,F_t^{-1}(\beta))= |\alpha - \beta| - |\alpha - \beta|^2.
\end{equation}
Thus for each $t \in E$ we have $d((t,F_t^{-1}(\alpha),(t,F_t^{-1}(\beta)) \rightarrow 0$ as $\alpha, \beta \rightarrow 0$ or $\alpha,\beta \rightarrow 1$. We also have
\begin{align*}
d((t,F_t^{-1}(\alpha),(t,F_t^{-1}(0)) &= \alpha -\alpha^2 \leq |\alpha-0|,\\
 d((t,F_t^{-1}(\alpha),(t,F_t^{-1}(1)) &= \alpha - \alpha^2 \leq |\alpha -1|,
\end{align*}
and
\[
 d((t,F_t^{-1}(0),(t,F_t^{-1}(1))=0,
\]
and hence the uniform continuity of $\widehat G$ along with $\widehat G(t, F_t^{-1}(0))\!=\! \widehat G(t, F_t^{-1}(1))$ $=0$ implies that uniformly in $t \in E$ we have $\widehat G_t(F_t^{-1}(\alpha))$ is uniformly continuous in $\alpha \in [0,1]$ with probability one. Moreover,  the process $\{\widehat G(t,x)\colon\  (t,x)$ $\in E \times \mathbb{R}\}$  has  separable support in $\ell_{\infty}(E\times \mathbb{R})$, and hence the upper cover used in (\ref{old 2.10}) is unnecessary as the function there is measurable. We also have with $\widehat{P}$-probability one that
\begin{equation}\label{old 2.12}
\sup_{t \in E, \alpha \in [0,1]} |\widehat G_t(F_t^{-1}(\alpha))| < \infty, 
\end{equation}
and hence $a_n$ converging to zero,  and  (\ref{old 2.7}) implies
\begin{equation}\label{old 2.13}
 \lim_{n \rightarrow \infty} \sup \{|u-\alpha|\colon\   t \in E, (\alpha, u) \in \Gamma_{n,t}\}=0. 
\end{equation}
Therefore, (\ref{old 2.10}) follows from (\ref{old 2.8}) and that uniformly in $t \in E$ we have  $\widehat G_t \circ F_t^{-1}(\alpha)$  uniformly continuous in $\alpha \in E  \times [0,1]$. Therefore, (\ref{old 2.10})  holds, and this implies (\ref{old 2.6}), so the lemma is proven.
\end{proof}

\subsection{Applying Lemma \ref{lem2.2} to an Empirical CLT}\label{sec2.1}

\indent 

 Assuming the empirical CLT holds over $\mathcal{C}$, the conclusions of Lemma~\ref{lem2.2} hold with $a_n= \frac{1}{\sqrt n}$, and we have proved the following lemma.

\begin{lem} \label{lem3.3}
For all $t \in E$ assume the distribution function $F(t,x)$ is strictly increasing and continuous
in $ x \in \mathbb{R}$, and that the CLT holds on $\mathcal{C}$ with limit $\{G(t,x)\colon\  (t,x) \in E \times \mathbb{R}\}$.
Then, with $\widehat{P}$-probability one, we have 
\begin{equation}\label{strongCLTeq}
\Big(\sup_{t \in E, \alpha \in[0,1]}|\sqrt n [(F_{t} \circ \widehat F_{n,t}^{-1})(\alpha) - I_t(\alpha)]+(\widehat G_t \circ F_t^{-1}(\alpha))|\Big)^{*} \rightarrow 0. 
\end{equation}
\end{lem}

Up to this point we have only assumed that the distribution functions  $\{F_t(\cdot)\colon\  t \in E\}$ are continuous and strictly increasing on $\mathbb{R}$, and that the empirical processes satisfy the CLT over $\mathcal{C}$. Now we add the assumptions that 
these distribution functions 
have  densities $\{f(t,\cdot)\colon\  t \in E\}$  such that 
\begin{equation}\label{unif-cont-densities}
 \lim_{\delta\to 0}\sup_{t \in E} \sup_{|u-v|\le\delta}|f(t,u) -f(t,v)| =0, 
\end{equation}
and for every closed interval $I$ in $(0,1)$ there is an $\theta(I)>0$ such that 
\begin{equation}\label{inf-eq}
\inf_{t \in E,\alpha \in I,|x-\tau_{\alpha}(t)|\le \theta(I)} f(t,x) \equiv c_{I,\theta(I)}>0. 
\end{equation}

\begin{lem}\label{quantile-LLN}
Assume for all $t \in E$ that the distribution functions $F(t,x)$ are strictly increasing and continuous, and that their densities $f(t,\cdot)$ satisfy (\ref{unif-cont-densities}) and (\ref{inf-eq}). If the CLT holds on $\mathcal{C}$, then
for every closed subinterval $I$ of $(0,1)$ 
\begin{equation}\label{quant-LLN}
\lim_{n \rightarrow \infty} [\sup_{t \in E, \alpha \in I}|\widehat{\tau}_{\alpha}^n(t) - \tau_{\alpha}(t)|]^{*}=0.
\end{equation}
in $\widehat{P}$ probability.
\end{lem}

\begin{proof} Since we are assuming (\ref{inf-eq}), fix $I$ a closed subinterval of $(0,1)$, and take $0<\epsilon<\theta(I)$. Let
\begin{equation}\label{quant-est-LLN}
A_n = \{ [\sup_{t \in E, \alpha \in I}|\widehat{\tau}_{\alpha}^n(t)-\tau_{\alpha}(t)|]^{*}>\epsilon\},
\end{equation}
and
\begin{equation}\label{emp-est-LLN}
B_n = \{ [\sup_{t \in E, x \in \mathbb{R}}|\widehat{F}_n(t,x)-F(t,x)|]^{*}>\delta\},
\end{equation}
where $0< \delta< \delta(\epsilon) \le \epsilon c_{I, \theta(I)}/2$. Then, since we have the CLT over $\mathcal{C}$ with respect to $\widehat{P}$,  Lemma 2.10.14 on page 194 of \cite{vw} implies there exists $n_{\delta}<\infty$ such that $n \ge n_{\delta}$ implies
\begin{equation}\label{B_n}
\widehat{P}(B_n) <\epsilon.
\end{equation}
In addition, (\ref{unif-cont-densities}) and (\ref{inf-eq}) imply we also have
\begin{equation}\label{emp-LLN-3}
\sup_{\alpha \in I, t \in E}F(t, \tau_{\alpha}(t)- \epsilon) < \alpha - \delta, 
\end{equation}
and 
\begin{equation}\label{emp-LLN-4}
\inf_{\alpha \in I, t \in E}F(t, \tau_{\alpha}(t) + \epsilon) > \alpha + \delta.
\end{equation}
That is, (\ref{emp-LLN-3}) holds by (\ref{inf-eq}) since
$$
\sup_{t \in E,\alpha \in I}F(t, \tau_{\alpha}(t)- \epsilon) \le \alpha - \inf_{t \in E,\alpha \in I}\int_{ \tau_{\alpha}(t)- \epsilon}^{\tau_{\alpha}(t)} f(t,x)dx,
$$
and if $\delta=\delta(\epsilon)\le \frac{\epsilon c_{I,\theta(I)}}{2}, 0<\epsilon \le \theta(I),$ we then have
$$
\inf_{t \in E, \alpha \in I}\int_{\tau_{\alpha}(t)- \epsilon}^{\tau_{\alpha}(t)} f(t,x)dx \ge \epsilon c_{I,\theta(I)}>\delta.
$$
Thus on $B_n^c$, for all $t \in E, \alpha \in I$,
$$
F(t, \widehat{ \tau}_{\alpha}^n(t)) \ge \widehat{F}_n(t,\widehat{\tau}_{\alpha}^n(t))- \delta  \ge \alpha -\delta,
$$
where the second inequality follows by definition of $\widehat{\tau}_{\alpha}^n(t)$. Combined with (\ref{emp-LLN-3}), on $B_n^c$ this implies that for all $t\in E,$ all $\alpha \in I$ 
\begin{equation}\label{LLN-ineq-1}
\widehat{\tau}_{\alpha}^n(t) \ge \tau_{\alpha}(t) - \epsilon.  
\end{equation}

Similarly, (\ref{emp-LLN-4}) holds by (\ref{inf-eq}) since
$$
\inf_{\alpha \in I, t \in E}F(t, \tau_{\alpha}(t) + \epsilon) = \alpha + \inf_{\alpha \in I, t \in E}\int_{ \tau_{\alpha}(t)}^{\tau_{\alpha}(t) + \epsilon} f(t,x)dx > \alpha +\delta,
$$
and if $\delta=\delta(\epsilon)\le \frac{\epsilon c_{I,\theta(I)}}{2}, 0<\epsilon \le \theta(I),$ we then have
$$
\inf_{\alpha \in I, t \in E}\int_{\tau_{\alpha}(t)}^{\tau_{\alpha}(t)+ \epsilon} f(t,x)dx \ge \epsilon c_{I,\theta(I)}>\delta.
$$
Thus for $ x < \widehat{\tau}_{\alpha}^n(t)$ and all $t\in E,\alpha \in I,$ on $B_n^c$ we have
$$
F(t,x) \le \widehat{F}_n(t,x) + \delta \le \alpha + \delta < F(t, \tau_{\alpha}(t) + \epsilon),
$$
where the first inequality follows from the definition of $B_n^c$, the second by definition of $\widehat{\tau}_{\alpha}^n(t)$ and that $x <\widehat{\tau}_{\alpha}^n(t)$, and the third by (\ref{emp-LLN-4}).
Thus
$\tau_{\alpha}(t) + \epsilon >x$ for all $ x < \widehat{\tau}_{\alpha}^n(t)$, and on $B_n^c$ we have
$$
\tau_{\alpha}(t) + \epsilon \ge  \widehat{\tau}_{\alpha}^n(t)
$$
for all $t \in E, \alpha \in I$. Combining this with (\ref{LLN-ineq-1}), on $B_n^c$  we have for all $t\in E,$ all $\alpha \in I,$ that
\begin{equation}\label{LLN-ineq-2}
\tau_{\alpha}(t)- \epsilon \le \widehat{\tau}_{\alpha}^n(t)  \le \tau_{\alpha}(t) +\epsilon.
\end{equation}
Hence on $B_n^c$
$$
 [\sup_{t \in E, \alpha \in I}|\widehat{\tau}_{\alpha}^n(t)-\tau_{\alpha}(t)|]^{*} \le \epsilon,
 $$
and on $B_n$ it is certainly bounded by $\infty$. Since $B_n$ is measurable, we thus have for $n \ge n_{\delta}$ that
$$
\widehat{P}( [\sup_{t \in E, \alpha \in I}|\widehat{\tau}_{\alpha}^n(t)-\tau_{\alpha}(t)|]^{*} > \epsilon) \le\widehat{P}(B_n) \le \epsilon.
$$
Since $\epsilon>0$ can be taken arbitrarily small, letting $n \rightarrow \infty$ implies (\ref{quant-LLN}). 
Thus the lemma is proven.
\end{proof}

\begin{thm}\label{density top}
Assume for all $t \in E$ that the distribution functions $F(t,x)$ are strictly increasing, their densities $f(t,\cdot)$ satisfy (\ref{unif-cont-densities}) and (\ref{inf-eq}), and the CLT holds on $\mathcal{C}$ with limit $\{G(t,x)\colon\  (t,x) \in E \times \mathbb{R}\}$.
Then, 
for $I$ a closed subinterval of $(0,1)$ we have
\begin{equation}\label{prob-CLT}
 \Big(\sup_{ t \in E, \alpha \in I}|\sqrt n (\widehat \tau_{\alpha}^n(t)- \tau_{\alpha}(t)) f(t, \tau_{\alpha}(t)) + \widehat G(t, \tau_{\alpha}(t)|\Big)^{*} \rightarrow 0 
\end{equation}
in $\widehat{P}$-probability, and therefore the quantile processes $\{\sqrt n (\widehat \tau_{\alpha}^n(t)- \tau_{\alpha}(t))$ $f(t, \tau_{\alpha}(t))\colon\  n \geq 1\}$ satisfy the CLT in $\ell_{\infty}(E \times I)$ with Gaussian limit process $\{ \widehat G(t,\tau_{\alpha}(t))\colon\  (t,\alpha) \in E \times I \}$. Moreover, the quantile processes $\{\sqrt n ( \widehat{\tau}_{\alpha}^n(t)- \tau_{\alpha}(t))\colon\  n \geq 1\}$ also satisfy the CLT in $\ell_{\infty}(E \times I)$ with Gaussian limit process $\{\frac{\widehat{G}(t,\tau_{\alpha}(t))}{ f(t, \tau_{\alpha}(t))}\colon\  (t,\alpha) \in E \times I \}$.
\end{thm}
\begin{proof} 
Applying Theorem 3.6.1 of \cite{Dudley-unif-clt}, the first CLT asserted follows immediately from (\ref{prob-CLT}). Hence we next turn to the proof of (\ref{prob-CLT}).

First we observe that under the given assumptions, we have (\ref{strongCLTeq}) holding. Furthermore, since the densities are assumed continuous, 
\begin{equation}\label{mvt-R}
F(t,y) -F(t,x) = f(t,x)(y-x) + R(t,x,y) (y-x), 
\end{equation}
where $R(t,x,y)=f(t,\xi(t)) - f(t,x)$, and $\xi(t)$ between $x$ and $y$ is determined by the mean value theorem applied to $F(t,\cdot)$. Of course, $R(t,\cdot,\cdot)$ depends on $F(t,\cdot)$, but we suppress that, and simply note that since $\xi(t)$ is between $x$ and $y$,
\begin{equation}\label{R-est}
|R(t,x,y)| \leq \sup_{u\in[x,y]\cup [y,x]} |f(t,u) -f(t,x)|. 
\end{equation}
Therefore, for $M>0$ we have
$$
\widehat{P}([\sup_{t \in E, \alpha \in I}|\sqrt n(\widehat{\tau}_{\alpha}^n(t) - \tau_{\alpha}(t))|\frac{f(t,\tau_{\alpha}(t))}{2}]^{*} \ge M) \le a_n(M) +b_n,
$$
where
\begin{align*}
a_n(M)&=\widehat{P}(A_{n,1} \cap A_{n,2}),\\
A_{n,1}&= \{[\sup_{t \in E, \alpha \in I}|\sqrt n(\widehat{\tau}_{\alpha}^n(t) - \tau_{\alpha}(t))|(f(t,\tau_{\alpha}(t))+ R(t,\widehat{\tau}_{\alpha}^n(t),\tau_{\alpha}(t)))]^{*} \ge M\},\\
A_{n,2}&=\{ [\sup_{t \in E, \alpha \in I}|R(t,\widehat{\tau}_{\alpha}^n(t),\tau_{\alpha}(t))|]^{*} \le \frac{ c_{I,\theta(I)}}{2}\},\\
b_n &=P([\sup_{t \in E, \alpha \in I}|R(t,\widehat{\tau}_{\alpha}^n(t),\tau_{\alpha}(t))|]^{*} > \frac{ c_{I,\theta(I)}}{2}),
\end{align*}
and 
$c_{I,\theta(I)}>0$ is given as in (\ref{inf-eq}).
Thus
$$
\widehat{P}([\sup_{t \in E, \alpha \in I}|\sqrt n(\widehat{\tau}_{\alpha}^n(t) - \tau_{\alpha}(t))|\frac{f(t,\tau_{\alpha}(t))}{2}]^{*} \ge M) \le P(A_{n,1})+b_n,
$$
and by (\ref{mvt-R}) we also have 
$$
A_{n,1}=\{[\sup_{t \in E, \alpha \in I}|\sqrt n(F_t(\widehat{\tau}_{\alpha}^n(t)) - F_t(\tau_{\alpha}(t)))|]^{*} \ge M\},
$$
which implies
\begin{align}\label{ineq-1}
&\widehat{P}([\sup_{t \in E, \alpha \in I}|\sqrt n(\widehat{\tau}_{\alpha}^n(t) - \tau_{\alpha}(t))|\frac{f(t,\tau_{\alpha}(t))}{2}]^{*} \ge M)\\
\qquad \le~ &\widehat{P}([\sup_{t \in E, \alpha \in I}|\sqrt n(F_t(\widehat{\tau}_{\alpha}^n(t)) - F_t(\tau_{\alpha}(t)))|]^{*} \ge M)+b_n.\notag
\end{align}

Since $I_t(\alpha)= F_t(F_t^{-1}(\alpha)), \alpha \in (0,1),t \in E,$ and $I \subseteq (0,1)$
\begin{align}\label{ineq-2}
&[\sup_{t \in E, \alpha \in I}|\sqrt n(F_t(\widehat{\tau}_{\alpha}^n(t)) - F_t(\tau_{\alpha}(t)))|]^{*}\\
\le~ &[\sup_{t \in E, \alpha \in[0,1]}|\sqrt n [(F_{t} \circ \widehat F_{n,t}^{-1})(\alpha) - I_t(\alpha)]+(\widehat G_t \circ F_t^{-1}(\alpha))|]^{*}\notag\\
 +~ &[\sup_{t\in E,\alpha \in I}|\widehat G_t \circ F_t^{-1}(\alpha))|]^{*}, \notag
\end{align}
and since the process $\{\widehat{G}(t,x):t \in E, x \in \mathbb{R}\}$ is sample continuous on $E\times \mathbb{R}$ in the semi-metric $d$ given in (\ref{eq1.1}) with Radon support in $\ell_{\infty}(E\times\mathbb{R})$ we also have
\begin{equation}\label{ineq-3}
[\sup_{t \in E, \alpha \in I}| \widehat G_t \circ F_t^{-1}(\alpha)|]^{*}= \sup_{t \in E, \alpha \in I}| \widehat G_t \circ F_t^{-1}(\alpha)|. 
\end{equation}
Therefore, for every $\epsilon>0$ and all $n \ge 1$, by combining (\ref{strongCLTeq}), (\ref{ineq-2}) and(\ref{ineq-3}) we have an $M=M(\epsilon)$ sufficiently large that
\begin{equation}\label{ineq-4}
 \widehat{P}([\sup_{t \in E, \alpha \in I}|\sqrt n(F_t(\widehat{\tau}_{\alpha}^n(t)) - F_t(\tau_{\alpha}(t)))|]^{*} \ge M) \le \epsilon.
\end{equation}

We now turn to showing that  
\begin{equation}\label{quant-bdd-prob}
[\sup_{t \in E, \alpha \in I}|\sqrt n  (\widehat {\tau}_{\alpha}^n(t)- \tau_{\alpha}(t))|]^{*} 
\end{equation}
is bounded in $\widehat{P}$-probability. That is, let
$$
\lambda(t,\delta) = \sup_{|u-v| \le \delta}|f(t,u)-f(t,v)|.
$$
Then
$$
|R(t,x,y)| \leq \lambda(t,|x-y|),
$$
and hence by (\ref{unif-cont-densities}) for every $\epsilon>0$ there exists $\delta>0$ such that $|x-y|\le \delta$ implies
$$
\sup_{t \in E}|R(t,x,y)| < \epsilon.
$$
Therefore, for every $\epsilon\in (0, \frac{c_{I,\theta(I)}}{2})$ there exists $\delta=\delta(\epsilon)>0$ suitably chosen such that
\begin{align}\label{R-bd-1}
b_n &\le \widehat{P}([\sup_{t \in E, \alpha \in I}|R(t,\widehat {\tau}_{\alpha}^n(t), \tau_{\alpha}(t))|]^{*} \ge \epsilon)\\
&= \widehat{P}^{*}(\sup_{t \in E, \alpha \in I}|R(t,\widehat {\tau}_{\alpha}^n(t), \tau_{\alpha}(t))| \ge \epsilon),\notag
\end{align}
and since
\begin{align}\label{R-bd-2}
\widehat{P}^{*}(\sup_{t \in E, \alpha \in I}|R(t,\widehat {\tau}_{\alpha}^n(t), \tau_{\alpha}(t))| \ge \epsilon) &\le \widehat{P}^{*}(\sup_{t \in E, \alpha \in I}|\widehat {\tau}_{\alpha}^n(t)- \tau_{\alpha}(t))| \ge \delta)\\
&= \widehat{P}([\sup_{t \in E, \alpha \in I}|\widehat {\tau}_{\alpha}^n(t)- \tau_{\alpha}(t))|]^{*} \ge \delta),\notag
\end{align}
Lemma \ref{quantile-LLN} implies for every $\epsilon \in (0, \frac{c_{I,\theta(I)}}{2})$ that
\begin{equation}\label{R-bd-3}
\lim_{n \rightarrow \infty} b_n=0.
\end{equation}
Combining (\ref{ineq-1}), (\ref{ineq-4}), and (\ref{R-bd-3}), we have (\ref{quant-bdd-prob}), i.e. $[\sup_{t \in E, \alpha \in I}|\sqrt n  (\widehat {\tau}_{\alpha}^n(t)- \tau_{\alpha}(t))|]^{*}$ 
is bounded in $\widehat{P}$-probability. Furthermore, we then also have that
\begin{equation}\label{R-bd-4}
[\sup_{t \in E, \alpha \in I}|\sqrt n  (\widehat {\tau}_{\alpha}^n(t)- \tau_{\alpha}(t))||R(t,\widehat{\tau}_{\alpha}^n(t), \tau_{\alpha}(t))|]^{*}
\end{equation}
converges in $\widehat{P}$ probability to zero.

Now, by (\ref{strongCLTeq}) and (\ref{mvt-R}) we have with $\widehat{P}$-probability one that
\begin{equation}\label{clt-R-bd}
\lim_{n \rightarrow \infty} \Big(\sup_{ t \in E, \alpha \in I}|\sqrt n (\widehat \tau_{\alpha}^n(t)- \tau_{\alpha}(t))[ f(t, \tau_{\alpha}(t))+ R(t,\tau_{\alpha}(t), \widehat \tau_{\alpha}^n(t))] + \widehat G(t, \tau_{\alpha}(t)|\Big)^{*} =0,
\end{equation}
and since
$$
[\sup_{ t \in E, \alpha \in I}|\sqrt n (\widehat \tau_{\alpha}^n(t)- \tau_{\alpha}(t))f(t, \tau_{\alpha}(t)) + \widehat G(t, \tau_{\alpha}(t)|]^{*}\leq u_n +v_n,  
$$
where 
$$
u_n \le [\sup_{ t \in E, \alpha \in I}|\sqrt n (\widehat \tau_{\alpha}^n(t)- \tau_{\alpha}(t))[ f(t, \tau_{\alpha}(t))+ R(t,\tau_{\alpha}(t), \widehat \tau_{\alpha}^n(t))] + \widehat G(t, \tau_{\alpha}(t)|]^{*}  
$$
and
$$
v_n \le[\sup_{ t \in E, \alpha \in I}|\sqrt n (\widehat \tau_{\alpha}^n(t)- \tau_{\alpha}(t))R(t,\tau_{\alpha}(t), \widehat \tau_{\alpha}^n(t))|] ^{*},
$$
we have by combining (\ref{R-bd-4}) and (\ref{clt-R-bd}) that
$$
[\sup_{ t \in E, \alpha \in I}|\sqrt n (\widehat \tau_{\alpha}^n(t)- \tau_{\alpha}(t))f(t, \tau_{\alpha}(t)) + \widehat G(t, \tau_{\alpha}(t)|]^{*}\rightarrow 0
$$
in $\widehat{P}$ probability. Hence (\ref{prob-CLT}) is proven.

To finish the proof it remains to check that the quantile processes
$$ 
\{\sqrt n ( \widehat{\tau}_{\alpha}^n(t)- \tau_{\alpha}(t))\colon\  n \geq 1\}.
$$ 
also satisfy the CLT in $\ell_{\infty}(E \times I)$ with Gaussian limit process $\{\frac{\widehat{G}(t,\tau_{\alpha}(t))}{ f(t, \tau_{\alpha}(t))}\colon\  (t,\alpha) \in E \times I \}$.
Since (\ref{prob-CLT}) holds, and by (\ref{inf-eq}) we have the non-random quantity
$$
\sup_{t \in E, \alpha \in I} \frac{1}{f(t, \tau_{\alpha}(t))} <\infty,
$$
we thus have
\begin{equation}\label{prob-CLT2}
 \Big(\sup_{ t \in E, \alpha \in I}|\sqrt n (\widehat \tau_{\alpha}^n(t)- \tau_{\alpha}(t))+ \frac{\widehat G(t, \tau_{\alpha}(t)}{f(t, \tau_{\alpha}(t))   }|\Big)^{*} \rightarrow 0 
\end{equation}
in $\widehat{P}$-probability.
The CLT  then follows from Theorem 3.6.1 of \cite{Dudley-unif-clt}, and that the Gaussian process $\widehat{G}$ is symmetric. Hence the theorem is proven.
\end{proof}

The next result shows that the conclusions of Theorem \ref{density top} also hold for the relevant processes as defined on the original probability space $(\Omega,\mathcal{A},Q)$. 

\begin{cor}\label{clt unhat}
Assume for all $t \in E$ that the distribution functions $F(t,x)$ are strictly increasing, their densities $f(t,\cdot)$ satisfy (\ref{unif-cont-densities}) and (\ref{inf-eq}), and the CLT holds on $\mathcal{C}$ with limit $\{G(t,x)\colon\  (t,x) \in E \times \mathbb{R}\}$.
Then, 
for $I$ a closed subinterval of $(0,1)$, 
the quantile processes $\{\sqrt n ( \tau_{\alpha}^n(t)- \tau_{\alpha}(t)) f(t, \tau_{\alpha}(t))\colon\  n \geq 1\}$ satisfy the CLT in $\ell_{\infty}(E \times I)$ with Gaussian limit process $\{ G(t,\tau_{\alpha}(t))\colon\  (t,\alpha) \in E \times I \}$. Moreover, the quantile processes $\{\sqrt n ( \tau_{\alpha}^n(t)- \tau_{\alpha}(t))\colon\  n \geq 1\}$ also satisfy the CLT in $\ell_{\infty}(E \times I)$ with Gaussian limit process $\{\frac{G(t,\tau_{\alpha}(t))}{ f(t, \tau_{\alpha}(t))}\colon\  (t,\alpha) \in E \times I \}$.
\end{cor}
 
\begin{proof} 
Recall the notation established at the start of this section in connection with the statement of Theorem \ref{thm2.1}, and the perfect mappings $g_n: \widehat{\Omega} \rightarrow \Omega$ such that $Q=\widehat{P} \circ g_n^{-1}$. In particular, equations (\ref{distribution-perfect}) to (\ref{tau}) are relevant.

For $u_1,\cdots,u_n \in D(E)$ and $n \geq 1,  t\in E, \alpha \in (0,1)$ define
\begin{align}\label{unhat-1}
 k_n(u_1,\cdots,u_n, t, \alpha) &=\sqrt n [ \inf \{ x: \sum_{j=1}^n I(u_j(t) \leq x) \geq n \alpha\}\\
&\quad - (F_t)^{-1}(\alpha)]f(t,\tau_{\alpha}(t)). \notag
\end{align}
where $\tau_{\alpha}(t)=  (F_t)^{-1}(\alpha)$.
Hence setting
$$
 r_n(t,\alpha,\omega) \equiv k_n(X_1,\cdots,X_n, t, \alpha)( \omega)\equiv k_n(X_1(\cdot, \omega),\cdots,X_n(\cdot,  \omega), t, \alpha),
$$
we then have 
\begin{equation}\label{unhat-2}
\sqrt n [(F_{n,t})^{-1}(\alpha)(\omega)- \tau_{\alpha}(t)]f(t,\tau_{\alpha}(t)) =  r_n( t, \alpha, \omega) 
\end{equation}
and
\begin{equation}\label{unhat-3}
\sqrt n [(\widehat {F}_{n,t})^{-1}(\alpha)(\hat \omega)- \tau_{\alpha}(t)]f(t,\tau_{\alpha}(t) ) =  r_n( t, \alpha,g_n(\hat \omega)) . 
\end{equation}
Therefore, for $ \hat \omega \in\widehat{ \Omega}$
$$
(\widehat{F}_{n,t})^{-1}(\alpha)( \hat \omega)= (F_{n,t})^{-1}(\alpha)( g_n( \hat\omega)),
$$
and for $h$ bounded on $\ell_{\infty}(E \times I)$ we have
$$
h(\sqrt n [(\widehat {F}_{n,t})^{-1}(\alpha)( \hat \omega)- \tau_{\alpha}(t)]f(t,\tau_{\alpha}(t)))= (h \circ r_n(t, \alpha, \cdot) \circ g_n)(\hat \omega),
$$
and hence the upper integrals
\begin{align*}
&\int_{\widehat{\Omega}}^{*} h(\sqrt n [(\widehat {F}_{n,t})^{-1}(\alpha)( \hat \omega)- \tau_{\alpha}(t)]f(t,\tau_{\alpha}(t)))d\widehat{P}\hat (\omega)\\
 =~ &\int_{\widehat{\Omega}}^{*}(h \circ r_n(t, \alpha, \cdot) \circ g_n)(\hat \omega)d\widehat{P}(\hat \omega)
\end{align*}
\begin{align}\label{unhat-4}
=  &\int_{\widehat{\Omega}}[(h \circ r_n(t, \alpha, \cdot) \circ g_n)]^{*}(\hat \omega)d\widehat{P}(\hat \omega)\\
= &\int_{\widehat{\Omega}}([h \circ r_n(t, \alpha, \cdot)]^{*} \circ g_n)(\hat \omega)d\widehat{P}(\hat \omega),\notag
\end{align}
where the last equality holds since $g_n$ is perfect. Now
\begin{align*}
 \int_{\widehat{\Omega}}([h \circ r_n(t, \alpha, \cdot)]^{*} \circ g_n)(\hat \omega)d\widehat{P}(\omega)&=  \int_{\Omega}[h \circ r_n(t, \alpha,  \omega)]^{*}dQ(\omega)\\
& = \int_{ \Omega}^{*}(h \circ r_n)(t, \alpha,  \omega)dQ( \omega),
\end{align*}
and therefore by (\ref{unhat-2}) and (\ref{unhat-4}), for all $h$ bounded on $\ell_{\infty}(E \times I)$,
\begin{align}
&\int_{\widehat{\Omega}}^{*} h(\sqrt n [(\widehat {F}_{n,t})^{-1}(\alpha)(\hat \omega)- \tau_{\alpha}(t)]f(t,\tau_{\alpha}(t)))d\widehat{P}(\hat\omega)\notag\\
\label{unhat-5}
 =  &\int_{\Omega}^{*}h(\sqrt n [(F_{n,t})^{-1}(\alpha)( \omega)- \tau_{\alpha}(t)]f(t,\tau_{\alpha}(t)))dQ( \omega). 
\end{align}
Now the equality in (\ref{unhat-5}) implies that the quantile processes
$$
\{\sqrt n [(\widehat {F}_{n,t})^{-1}(\alpha)( \hat \omega)- \tau_{\alpha}(t)]f(t,\tau_{\alpha}(t) ): n\geq 1, t \in E, \alpha \in I\}
$$
satisfy the $CLT$ in $\ell{_\infty}(E \times I)$ if and only if 
$$ 
\{\sqrt n [(F_{n,t})^{-1}(\alpha)( \omega)- \tau_{\alpha}(t)]f(t,\tau_{\alpha}(t) ): n\geq 1, t \in E, \alpha \in I\}
$$
satisfy the $CLT$ there, and they have the same Gaussian limit, namely 
$$
\{G(t,\tau_{\alpha}(t)): t \in E, \alpha \in I\}.
$$

A similar argument implies the quantile processes
$$
\{\sqrt n [(\widehat {F}_{n,t})^{-1}(\alpha)( \hat \omega)- \tau_{\alpha}(t)]: n\geq 1, t \in E, \alpha \in I\}
$$
satisfy the $CLT$ in $\ell{_\infty}(E \times I)$ if and only if 
$$ 
\{\sqrt n [(F_{n,t})^{-1}(\alpha)( \omega)- \tau_{\alpha}(t)]: n\geq 1, t \in E, \alpha \in I\}
$$
satisfy the $CLT$ there, and they have the same Gaussian limit. Since Theorem \ref{density top} implies the
Gaussian limit of 
$$
\{\sqrt n [(\widehat {F}_{n,t})^{-1}(\alpha)( \hat \omega)- \tau_{\alpha}(t)]: n\geq 1, t \in E, \alpha \in I\}
$$
is given by  
$$
\{\frac{\widehat{G}(t,\tau_{\alpha}(t))}{f(t,\tau_{\alpha}(t))}: t \in E, \alpha \in I\},
$$
which has the same Radon law on $\ell_{\infty}(E \times I)$ as
$$
\{\frac{G(t,\tau_{\alpha}(t))}{f(t,\tau_{\alpha}(t))}: t \in E, \alpha \in I\},
$$
the corollary is proven.

\end{proof}

\section{The Empirical CLT over $\mathcal{C}$}\label{applications}

In order to prove the empirical quantile CLT of Theorem~\ref{density top}, and its corollary,  we assumed the empirical CLT over $\mathcal{C}$ holds. Empirical results of this type were established in \cite{kkz} for fractional Brownian motions and the Brownian sheet as long as these processes were not fixed to be zero at some point, and later in the paper we will use these facts to establish the empirical quantile CLT with those processes as the base process. The purpose of this section is to broaden the class of base processes to which \cite{kkz} applies, and that then will also be potential applications for our quantile process results. In particular, in this section we show how \cite{kkz} yields the empirical CLT for a broader  class of Gaussian processes, all compound Poisson processes, and also many stationary independent increment processes and martingales. In particular, these results apply to all symmetric stable processes, and below we will show that under certain ci
 rcumstances empirical quantile CLT's also hold for such processes. 

In the typical empirical process result over $\mathcal{C}$ that we establish, the input process $\{X(t): t \in E\}$ is given in terms of a base process $\{Y(t): t \in E \}$ where $P(Y(0)=0)=1$ when $E=[0,T]$, and $X(t)=Y(t)+Z, t \in E$. The random variable $Z$ is assumed independent of the base process, and has a density which is uniformly bounded on $\mathbb{R}$, or in $L_a(\mathbb{R})$ for some $a \in (1,\infty)$. The use of $Z$ allows to say the densities of each $X(t)$ have a uniform property provided the density of $Z$ has that property, and is an efficient way to do this. More important, however, is that in many classical examples the base process
 $\{Y(t): t \in E\}$  with $P(Y(0)=0)=1$ fails the empirical CLT over $\mathcal{C}$, yet the input process $X(t)=Y(t)+Z, t \in E,$ satisfies it. Examples of this type include fractional Brownian motions on $E=[0,T]$, the d-parameter Brownian sheet on $E=[0,T]^d$, and also strictly stable processes with stationary independent increments. This was pointed out for fractional Brownian motions in \cite{kkz}, and we will say more about the other examples at appropriate points of this section.
   
\subsection{ Additional Gaussian process empirical CLT's over $\mathcal{C}$ }
Throughout this subsection we assume $E$ is a compact subset of the d-fold product of $[0,T]$, which we denote by $[0,T]^d$, and that $\{X_t: t \in E\}$ is a centered Gaussian process whose $L_2$-distance $d_X$ is such that for some $k_1<\infty, s,t \in E$,
\begin{equation}\label{euclidean-eq}
d_X(s,t)= [\E((X_t-X_s)^2)]^{\frac{1}{2}} \leq k_1e^{\gamma}(s,t),
\end{equation}
where $e(s,t)$ is the usual $L_2$-distance on $\mathbb{R}^d$ and $0< \gamma \leq1$ . Furthermore, applying Theorem 6.11 on page 144 of \cite{araujo-gine-book} we have from (\ref{euclidean-eq}) that $\{X_t: t \in E\}$ has a sample continuous version $\{\tilde X_t: t \in E\}$ such that for $s,t \in E$
\begin{equation}\label{eq3.2}
| \tilde X_t - \tilde X_s | \leq \Gamma e^r(s,t), 
\end{equation} 
where $ \Gamma < \infty$ with probability one, and $0<r < \gamma$. Hence, without loss of generality, we may also assume throughout the sub-section that $\{X_t: t \in E\}$ is sample continuous with (\ref{eq3.2})
holding.

\begin{prop}\label{prop3.1}
Let $E$ be a compact subset of $[0,T]^d$, and assume $\{X_t: t \in E\}$ is a sample continuous centered Gaussian process such that (\ref{euclidean-eq}) holds and for all $x,y \in \mathbb{R}, k_2<\infty,$ and some $\beta \in(0,1]$
\begin{equation}\label{eq3.3}
\sup_{t \in E} |  F_t(x) - F_t(y) | \leq k_2 |x-y|^{\beta}. 
\end{equation}  
Then, the empirical CLT built from the process $\{X_t: t \in E\}$ holds over $\mathcal{C}$. Moreover, if $\{Y_t: t \in E\}$ is a sample continuous centered Gaussian process such that (\ref{euclidean-eq}) holds, and $Z$ is a random variable independent of $\{Y_t: t \in E\}$ whose density is uniformly bounded on $\mathbb{R}$, or in $L_p(\mathbb{R})$ for some $p \in (1,\infty)$, then  the empirical CLT based on the process $\{X_t: t \in E\}$ holds over $\mathcal{C}$, where $X_t=Y_t +Z, t \in E$.
\end{prop}

\begin{proof} First we assume $\{X_t: t \in E\}$ is a sample continuous centered Gaussian process such that (\ref{euclidean-eq}), (\ref{eq3.2}), and (\ref{eq3.3}) hold. Then, applying the Fernique-Landau-Shepp result we have exponential decay of the tail probability of $\Gamma$ in (\ref{eq3.2}), and hence assumptions (I ) and (II)
of Theorem 5 in \cite{kkz} hold. If (\ref{euclidean-eq}) and (\ref{eq3.2}) hold for $\{Y_t: t \in E\}$ and $X_t=Y_t + Z$, where the density of $Z$ is uniformly bounded density or in $L_p$ as indicated,  then standard convolution formulas imply (\ref{eq3.3}) holds for $\{X_t: t \in E\}$. In particular, if the  density is $Z$ is assumed to be uniformly bounded, then (\ref{eq3.3}) holds with $\beta=1$, and if it is in $L_p(\mathbb{R})$, then $\beta=1-1/p$ suffices. Therefore, under either assumption on the density of $Z$, we have assumptions (I ) and (II)
of Theorem 5 in \cite{kkz} holding for $\{X_t:t \in E\}$. 

Therefore, the conclusions of the proposition hold in either situation provided we verify condition (III) of that theorem. That is, from Remark 8 in \cite{kkz} we need to verify there exists a centered Gaussian process $\{H^{\alpha}(t):t \in E\}$ with $L_2$-distance $\rho_{\alpha}(s,t)$, which is sample bounded and uniformly continuous on $(E, \rho_{\alpha})$, and for some $\alpha \in(0,  \frac{\beta}{2})$ we have
\begin{equation}\label{eq3.4}
(e(s,t))^{r\alpha}\leq  \rho_{\alpha}(s,t), s,t \in E.
\end{equation}

To verify (\ref{eq3.4}) we first assume $d=1$, and hence that $E$ is a compact subset of $[0,T]$. The Gaussian process $\{H^{\alpha}(t): t \in E\}$ is then defined to be a centered, sample continuous, $\alpha \theta$ fractional Brownian motion on $[0,T]$ with $L_2$-distance $\rho_{\alpha}(s,t) =|s-t|^{\alpha\theta}, s,t \in [0,T]$ and $\theta \in (0,r)$ sufficiently small that (\ref{eq3.4}) holds.

If $d>1$ and $\alpha \in (0, \frac{\beta}{2})$ is fixed, then $E$ is a compact subset of $[0,T]^d$, and the Gaussian process $\{H^{\alpha}(t): t \in E\}$ is defined to be the sum 
\begin{equation}
H^{\alpha}(t) = \sum_{j=1}^d H^{\alpha} _j(t_j), t=(t_1,\cdots,t_d) \in [0,T]^d,
\end{equation}
where the $H^{\alpha}_j$ are 
centered, independent sample continuous, $\alpha \theta$ fractional Brownian motions on $[0,T]$ such that for $t_j\geq 0$ and $\theta \in (0,r)$
\begin{equation}
\E((H^{\alpha}_j(t_j))^2)= (T\vee 1)^{2r\alpha}t_j^{2\theta\alpha}.
\end{equation}
Hence the $L_2$-distance for $\{H^{\alpha}(t): t \in [0,T]^d\}$ is
\begin{equation}\label{rho alpha}
\rho_{\alpha}(s,t) =(\sum_{j=1}^d (T \vee 1)^{2r\alpha} |t_j-s_j|^{2\alpha\theta})^{\frac{1}{2}}, s,t \in [0,T]^d,
\end{equation}
and with $e(s,t)$ the Euclidean distance on $\mathbb{R}^d$, we have
\begin{equation}
e(s,t)^{r\alpha} = (\sum_{j=1}^d |t_j-s_j|^2)^{\frac{r\alpha}{2}} \leq (\sum_{j=1}^d (\frac{ |t_j-s_j|}{T\vee 1})^{2r\alpha} (T\vee 1)^{2r\alpha})^{\frac{1}{2}},  
\end{equation}
where the  inequality holds since $0< r\alpha < 1$. Since $\theta \in (0,r)$ we therefore have 
\begin{equation}\label{eralpha}
e(s,t)^{r\alpha} = (\sum_{j=1}^d |t_j-s_j|^2)^{\frac{r\alpha}{2}} \leq (\sum_{j=1}^d (\frac{ |t_j-s_j|}{T\vee 1})^{2\alpha\theta} (T\vee 1)^{2r\alpha})^{\frac{1}{2}}. 
\end{equation}
Therefore, by combining   (\ref{rho alpha}), and (\ref{eralpha}) we have  (\ref{eq3.4}), and the proof is complete.
\end{proof}

Our first application of Proposition~\ref{prop3.1} is to fractional Brownian motions. This result was obtained in \cite{kkz}, but we include it here as its proof is an immediate application of this proposition, and an empirical quantile CLT result will also be obtained for these processes later in the paper.
\begin{cor}\label{cor3.1}
Let $E=[0,T]$, and assume $\{Y_t\colon\  t \in E]\}$ is a centered sample continuous $\gamma$-fractional Brownian motion for $0<\gamma< 1$ such that $Y_0=0$ with probability one and $\E(Y_t^2)= t^{2\gamma}$ for $t \in E$. Set $X_t = Y_t + Z$, where $Z$ is independent of  $\{Y_t\colon\  t \in E\}$, and assume $Z$ has a density that is uniformly bounded on $\mathbb{R}$ or is in $L_p(\mathbb{R})$ for some $p \in (1, \infty)$.
Then, the empirical CLT holds over $\mathcal{C}$. 
\end{cor}

\begin{proof} The $L_2$-distance for  $\{X_t\colon\  t \in E\}$ is $d_X(s,t) = |s-t|^{\gamma}$, and hence (\ref{euclidean-eq}) holds with $k-1=1$. Also, (\ref{eq3.2}) holds with $0< r< \gamma$, and the assumptions on the density of $Z$ then imply (\ref{eq3.3}). Therefore, Proposition~\ref{prop3.1} applies to complete the proof.
\end{proof}

Our next application of Proposition~\ref{prop3.1} is to the d-dimensional Brownian sheet. A result for $d=2$ appeared in \cite{kkz}, but once we have Proposition~\ref{prop3.1} in hand, the d-dimensional case follows easily.

\begin{cor}\label{cor3.2}
Let $E=[0,T]^d$ for $d \geq 2$, and assume  $\{Y_{t}\colon\  t \in E\}$ is a centered sample continuous Brownian sheet with covariance function 
\begin{equation}\label{sheet covariance}
\E(Y_sY_t)= \prod_{j=1}^d (s_j \wedge t_j),~ s=(s_1,\cdots,s_d),t=(t_1,\cdots,t_d) \in E.
\end{equation}
For $t \in E$, let $X_t = Y_t + Z$, where $Z$ is  independent of  $\{Y_{t}\colon\  t \in E\}$, and assume $Z$ has a density that is uniformly bounded on $\mathbb{R}$ or is in $L_p(\mathbb{R})$ for some $p \in (1, \infty)$. Then,  the empirical CLT based on the process $\{X_t: t \in E\}$ holds over $\mathcal{C}=\{C_{t,x}\colon\  (t,x) \in E \times \mathbb{R}\}$, where in this setting $C_{t,x} = \{ z\in D(E)\colon\  z(t) \leq x\}$, and $D(E)$ denotes the continuous functions on $E$. Moreover, the empirical CLT over $\mathcal{C}$ fails for the base process $\{Y(t): t \in E\}$.
\end{cor}

\begin{proof} 
First we observe that if $0\leq s_j \leq t_j \leq T$ for $j=1,\cdots,d$, then for $d \geq 1$
\begin{equation}\label{product difference}
\prod_{j=1}^d t_j - \prod_{j=1}^d s_j \leq T^{d-1} \sum_{j=1}^d |t_j- s_j|.
\end{equation}
This elementary fact is obvious for $d=1$ with $T^0=1$, and for $d \geq 2$ it follows by an easy induction argument.
Moreover, the $L_2$-distance for $\{X_t:t \in E\}$ satisfies

$$
d^2_X(s,t) = |\prod_{j=1}^d t_j +\prod_{j=1}^d s_j -2\prod_{j=1}^d (s_j \wedge t_j) | ,
$$
and hence
$$
d^2_X(s,t) \leq T^{d-1}[\sum_{j=1}^{d} (t_j- t_j\wedge s_j) +\sum_{j=1}^{d} (s_j- t_j\wedge s_j)]= T^{d-1}\sum_{j=1}^{d} |t_j-  s_j|.
$$
Therefore,
$$
d^2_X(s,t) \leq dT^{d-1}\sum_{j=1}^{d} |t_j-  s_j|\frac{1}{d} \leq dT^{d-1}(\sum_{j=1}^{d} |t_j-  s_j|^2\frac{1}{d})^{\frac{1}{2}} = d^{\frac{1}{2}}T^{d-1}e(s,t),
$$
which implies $ d_X(s,t) \leq T^{\frac{d-1}{2}} d^{\frac{1}{4}}e^{\frac{1}{2}}(s,t)$, and hence (\ref{euclidean-eq}) holds. Either assumption for the density of $Z$ implies (\ref{eq3.3}) for a suitable $\beta$, and thus
Proposition~\ref{prop3.1} applies to show the CLT over $\mathcal{C}$ holds for $\{X(t): t \in [0,T]^d\}$ holds.

To see  why this CLT fails for the base process $\{Y(t): t \in [0,T]^d\}$, observe that the process $W(r)= Y(r^{\frac{1}{d}}(1,\cdots,1)), 0 \leq r \leq T^d,$ is a Brownian motion with $P(W(0)=0)=1$. Thus by Lemma 5 of \cite{kkz} we have $\{Y(t):t \in [0,T]^d\}$ fails the CLT over the the class of sets $\mathcal{C}_1= \{C_{r,x}: 0 \leq r \leq T^d, x \in \mathbb{R}\}$, where $C_{r,x}= \{ z \in D(E): z(r^{\frac{1}{d}}(1,\cdots,1))\leq x \}$. Since $\mathcal{C}_1 \subseteq \mathcal{C}$, it follows from that the CLT  for $Y$ over $\mathcal{C}$ must also fail.
\end{proof}

\subsection{ Compound Poisson process empirical CLT's over $\mathcal{C}$}
Here we examine the empirical CLT over $\mathcal{C}$ when our base process is an arbitrary compound Poisson process. This will be done in the next proposition by applying Theorem 3 of \cite{kkz}. 
We will see from its proof that the Gaussian process needed for this application can be taken to be a sample continuous Brownian motion, and the space of functions $D(E)$, when $E=[0,T]$, is the standard $D$-space of functions on $[0,T]$ which are right continuous on$[0,T)$ and have left limits on $(0,T]$. These examples  are somewhat surprising since the sample paths of the base process $\{Y(t)\colon\   t \in [0,T]\}$ have jumps, while those of significance in \cite{kkz} and the previous subsection were all sample path continuous. 

To define the base process  in these examples we let  $\{N(t)\colon\  0 \leq t < \infty\}$ be a Poisson process with parameter $\lambda \in (0,\infty)$, and jump times $\tau_1,\tau_2,\cdots$. As usual we assume $P(N(0)=0)=1$, and that the sample paths $\{N(t)\colon\  0 \leq t \leq \infty\}$  are right continuous and  nondecreasing. Also, let $\{Y_k\colon\  k \geq 1\}$
be i.i.d. real-valued random variables, independent of $\{N(t)\colon\  0 \leq t \leq \infty\}$, and without mass at zero. Then, $Y(t)$ is defined to be
zero on $[0,\tau_1)$, $Y_1$ on $[\tau_1,\tau_2)$, and $Y_1+\cdots + Y_k$ on $[\tau_k, \tau_{k+1})$ for $k \geq 1.$

\begin{prop}\label{prop3.2}
The empirical process built from i.i.d.\ copies of the compound Poisson process  $\{Y(t)\colon\   t \in E \}$ with parameter $\lambda \in (0,\infty)$ and $E=[0,T]$ satisfies the CLT over $\mathcal{C}$.
\end{prop}

\begin{proof} 
The proof follows by applying Theorem 3 of \cite{kkz}. This is accomplished by showing  $\{Y(t)\colon\   t \in E \}$ satisfies the $L$ condition of \cite{kkz} when the Gaussian process involved is Brownian motion and the $\rho$ distance is a multiple of standard Euclidean distance on $[0,T]$. Since the distribution function of $Y(t)$ is not necessarily continuous, the $L$-condition involves distributional transforms of the the distribution functions 
$F(t,x)=P(Y(t) \leq x)$ denoted by $\tilde F_t(x)$. They are defined for $t \in E,x \in \mathbb{R}$ as
$$
\tilde F_t(x)= F(t,x^-) +V(F(t,x)-F(t,x^-)),
$$
where $V$ is a uniform random variable on $[0,1]$ independent of the process $\{Y(t)\colon\  t \in E \}$.  

To verify the $L$-condition for the $Y$ process, let $\{H(t)\colon\ 0 \leq t < \infty\}$ be a sample continuous Brownian motion with $P(H(0)=0)=1$ satisfying
$$
\rho^2(s,t) =\E((H(s)-H(t))^2) = 4(\lambda\vee1)|t-s|.
$$
Then, for $\epsilon >0$
\begin{align*}
\Lambda &\equiv \sup_{t \in [0,T]} P\Big(\sup_{\{s\colon\ \rho (s,t) \le \epsilon\}}|\tilde{F}_t(Y(s)) - \tilde{F}_t(Y(t))| > \epsilon^2\Big)\\
& \leq \sup_{t \in [0,T]} P\Big(\sup_{\{s\colon\ \rho (s,t) \le \epsilon\}}|Y(s) - Y(t)| > 0\Big)\\
&= \sup_{t \in [0,T]} \Big[1 -P\Big(\sup_{\{s\colon\ \rho (s,t) \le \epsilon\}}|Y(s) - Y(t)| = 0\Big)\Big].
\end{align*}
Since $Y(s)-Y(t) =0$ whenever $N(s)-N(t)=0$, and for $t \in [0,T]$ fixed
\begin{align*}
&\left\{\sup_{\{s\colon\ \rho (s,t) \le \epsilon\}}|N(s) - N(t)| = 0\right\}\\
 =~ &\left\{N\left(\left(t+\frac{\epsilon^2}{4(\lambda \vee 1)}\right)\wedge T\right) -N\left(\left(t-\frac{\epsilon^2}{4(\lambda \vee 1)}\right)\vee0\right)=0\right\},
\end{align*}
it follows that
\begin{align*}
P\bigg( N\left(\left(t+\frac{\epsilon^2}{4(\lambda \vee 1)}\right)\wedge T\right) &-N\left(\left(t-\frac{\epsilon^2}{4(\lambda \vee 1)}\right)\vee0\right)=0\bigg)\\
&= P(\sup_{\{s\colon\ \rho (s,t) \le \epsilon\}}|N(s) - N(t)| = 0)\\
&\leq P(\sup_{\{s\colon\ \rho (s,t) \le \epsilon\}}|Y(s) - Y(t)| = 0).
\end{align*}
Now
\begin{align*}
&P\left( N\left(\left(t+\frac{\epsilon^2}{4(\lambda \vee 1)}\right)\wedge T\right) -N\left(\left(t-\frac{\epsilon^2}{4(\lambda \vee 1)}\right)\vee0\right)=0\right)\\\
&\quad \geq \exp\left\{- \frac{\lambda \epsilon^2}{2(\lambda\vee1)}\right\},
\end{align*}
and hence for $0< \epsilon < \epsilon_0$ we have
$$
\Lambda \leq 1-\exp\left\{-\frac{ \epsilon^2}{2}\right\} \leq \epsilon^2.
$$
Taking $L$ suitably large we have for all $\epsilon>0$ that $\Lambda  \leq L\epsilon^2$, and hence the $L$-condition holds for the compound Poisson process $Y$, completing the proof of the proposition.
\end{proof}

\begin{rem}\label{rem3.1}
Let $Z$ be a random variable independent of both $\{N(t)\colon\  0 \leq t < \infty\}$ and $\{Y_k\colon\  k \geq 1\}$, and define
$$
X(t)=Y(t) +Z, t \in [0,T].
$$
Since the $L$-condition for the input process $\{X(t)\colon\  t \in E\}$ involves only its increments, and those are identical to those of the base process $\{Y(t)\colon\  t \in E\}$,
the argument above implies the $L$-condition also  holds for  $\{X(t)\colon\  t \in E\}$ Therefore, the empirical process built from i.i.d.\ copies of $X$ satisfies the CLT on $\mathcal{C}$.
\end{rem}

\subsection{Empirical process CLT's over $\mathcal{C}$ for other independent increment processes and martingales}

The processes we study here are either martingales, or stationary independent increment processes. There is some overlap  between these examples and the compound Poisson processes of the previous sub-section, as such processes have stationary independent increments, and could also be martingales. However, it is easy to check that there are examples which fit into one and only one of the classes we study. 

Let $E=[0,T]$, and assume $\{Y(t): t \geq 0 \}$ is a stochastic process whose sample paths are right continuous , with left hand limits on $[0,\infty)$, and satisfying $P(Y(0)=0)=1$.We say that  $\{Y(t): t \geq 0\}$ has $L_p$-increments that are
Lip-$\beta$ on $E$ if for some $p \in (0,1]$ and for
all $s,t \in E$ there is a $\beta \in (0,1]$ and $C<\infty$ such that
\begin{equation}\label{eq3.12}
\E(|Y(t)-Y(s)|^p) \leq C|t-s|^{\beta}. 
\end{equation}
For example, if  $\{Y(t): t \geq 0 \}$ is a strictly stable process with stationary independent increments and index $r \in (0,2]$ , then for $r \in (1,2]$ we have $\E(|Y(t)|)= t^{\frac{1}{r}}\E(|Y(1)|)$ and hence 
$$
\E(|Y(t)-Y(s)|) = |t-s|^{\frac{1}{r}}\E(|Y(1)|),
$$
which implies it has $L_1$-increments that are Lip-${\frac{1}{r}}$. Of course, it is also a martingale when $r \in (1,2]$. If $0<r \leq 1$ , then for $0<p<r$ we have
$$
\E(|Y(t)-Y(s)|^p) = |t-s|^{\frac{p}{r}}\E(|Y(1)|^p),
$$ 
which implies it has $L_p$-increments that are Lip-${\frac{p}{r}}$.
 
If $\{Y(t): t \geq 0\}$ is a square integrable martingale with $\lambda(t)=\E(Y^2(t)), t  \geq 0,$ then for  $0 \leq s \leq t$ the orthogonality of the increments of $\{Y(t): t \geq 0\}$  implies
\begin{equation}\label{eq3.13}
\E((Y(t)-Y(s))^2) = \lambda (t) -\lambda(s). 
\end{equation}
Hence, if $\lambda(\cdot)$ is Lip-$\gamma$ on $E$, then (\ref{eq3.13}) implies (\ref{eq3.12}) with $p=1, \beta=\gamma/2$. In addition, if $\{Y(t): t \geq 0 \}$ also has stationary, independent increments with $P(Y(0) =0)=1$ and
$\lambda(t) = \E(|Y(t)|)<\infty, t \geq 0,$ then for $ s,t \in E$ we have
\begin{equation}\label{eq3.14}
\E(|Y(t)-Y(s)|) =\E(|Y(|t-s|)|) \leq \lambda (|t-s|).
\end{equation}
Therefore, if $\lambda(t) \leq Ct^{\beta}$ for $t \in [0,\delta]$ and  some $\delta>0, \beta \in (0,1]$, then it is easy to check that (\ref{eq3.14}) implies (\ref{eq3.12}) with $p=1$ and the given $\beta$ for all $s,t \in E$, and a possibly larger constant C.

We also assume $Z$ is a random variable independent of $\{Y(t): t \geq 0 \}$ with   density $g(\cdot)$ on $\mathbb{R}$ such that
\begin{equation}\label{eq3.15}
k= \sup_{x \in \mathbb{R}} |g(x)|< \infty~\rm {or}~ g \in L_a(\mathbb{R})
\end{equation}
for some $a \in (1,\infty)$.
Let $X(t)= Z + Y(t), t \geq 0$, and denote the distribution function of $X(t), t \geq 0$ by $F_t(x)$. Then, if $g$ is uniformly bounded
\begin{equation}\label{eq3.16}
 \sup_{ t \in E}|F_t(x)- F_t(y)| \leq k|x-y|, x,y \in \mathbb{R},
\end{equation}
and if $g \in L_a(\mathbb{R})$ we have a $\tilde k< \infty$
\begin{equation}\label{eq3.17} 
\sup_{ t \in E}|F_t(x)- F_t(y)| \leq \tilde k |x-y|^{1-\frac{1}{a}}, x,y \in \mathbb{R}.
\end{equation}

\begin{prop}\label{prop3.3}
Let $E=[0,T]$, and assume $\{Y(t): t \geq 0 \}$ is a stochastic process whose sample paths are right continuous, with left hand limits on $[0,\infty)$, and satisfying $P(Y(0)=0)=1$. Furthermore, assume
 $\{Y(t)\colon\   t \in E \}$ is a martingale whose $L_1$-increments are Lip-$\beta$ for some $\beta \in (0,1]$ , or a stationary independent increments process satisfying (\ref{eq3.12}) for some $p \in (0,1)$ and $\beta \in (0, 1]$. Let $X(t)= Z + Y(t), t \geq 0$, where  $Z$ is a random variable independent of $\{Y(t): t \geq 0 \}$ and having  density $g(\cdot)$ on $\mathbb{R}$ satisfying (\ref{eq3.15}).
Then, the empirical process built from i.i.d. copies of  $\{X(t)\colon\   t \in E \}$ satisfies the CLT over $\mathcal{C}$.
\end{prop}
 
\begin{proof} 
Let $\rho(s,t)=|s-t|^{\theta}$. Then, $\rho$ is the $L_2$-distance of a $\theta$-fractional Brownian motion on $E$, and the proposition follows from Theorem 3 of \cite{kkz} provided we verify the $L$-condition for $\{X(t)\colon$ $t \in E\}$ with respect to $\rho$ and an appropriately chosen $\theta$. That is, since the distribution functions $F_t(\cdot)$ have a density, they are continuous, and hence it suffices to show for an appropriate $\theta>0$ 
there is a constant $L < \infty$ such that for every $\epsilon> 0$
\begin{equation}\label{eq3.18}
\sup_ {t \in E} P(\sup_{\{s:s \in E, \rho(s,t) \leq \epsilon\}}|F_t(X_s) -F_t(X_t)| > \epsilon^2) \leq L\epsilon^2.
\end{equation}

We prove the $L$-condition holds assuming the density $g$ of $Z$ is uniformly bounded, and hence we have (\ref{eq3.16}) holding. The proof when $g \in L_a(\mathbb{R})$ is essentially the same, only the algebra changes, and hence the details are left to the reader.

First we examine the situation when $\{Y(t): t \geq 0 \}$ is a martingale satisfying (\ref{eq3.12}) with $p=1$ and some $\beta \in (0,1]$. Applying (\ref{eq3.16}) to (\ref{eq3.18}) we then have
\begin{equation}\label{eq3.19}
\sup_ {t \in E} P(\sup_{\{s:s \in E, \rho(s,t) \leq \epsilon\}}|F_t(X_s) -F_t(X_t)| > \epsilon^2)  \leq A_{\epsilon} + B_{\epsilon}, 
\end{equation}
where
$$
A_{\epsilon} = \sup_{ t \in E} P(\sup_{\{s: s \in [t, (t+\epsilon^{\frac{1}{\theta}})\wedge T]\}}|X_s -X_t| > \frac{\epsilon^2}{2k}),
$$
and
$$
B_{\epsilon}= \sup_{t \in E}P(\sup_{\{s: s \in [ (t-\epsilon^{\frac{1}{\theta}})\vee 0, t ]\}}|X_s -X_t| > \frac{\epsilon^2}{2k}).
$$
Now
$$
A_{\epsilon} = \sup_{ t \in E} P(\sup_{\{s: s \in [t, (t+\epsilon^{\frac{1}{\theta}})\wedge T]\}}|Y_s -Y_t| > \frac{\epsilon^2}{2k}),
$$
and hence Doob's martingale maximal inequality implies
\begin{equation}\label{eq3.20}
A_{\epsilon} \leq \sup_{ t \in E} 2k\epsilon^{-2} \E(|Y_{(t + \epsilon^{\frac{1}{\theta}}) \wedge T} -Y_t|) \leq 2kC\epsilon^{-2+\frac{\beta}{\theta}},
\end{equation}
where the last inequality follows from (\ref{eq3.12}) with $p=1$. We also have
\begin{align*}
B_{\epsilon}&\leq  \sup_{t \in E}P(|Y_{ (t-\epsilon^{\frac{1}{\theta}})\vee 0 } -Y_t| > \frac{\epsilon^2}{4k})\\
&\quad+ \sup_{t \in E}P(\sup_{\{s: s \in [ (t-\epsilon^{\frac{1}{\theta}})\vee 0, t ]\}}|Y_s - Y_{ (t-\epsilon^{\frac{1}{\theta}})\vee 0 }| > \frac{\epsilon^2}{4k}),
\end{align*}
and using Markov's inequality, the martingale maximal inequality, and (66) with $p=1$ as before, we have
\begin{equation}\label{eq3.21}
B_{\epsilon}\leq 8kC\epsilon^{-2 +\frac{\beta}{\theta}}. 
\end{equation}
Combining (\ref{eq3.19}),(\ref{eq3.20}), and (\ref{eq3.21}) we have 
\begin{equation}\label{eq3.22}
\sup_ {t \in E} P(\sup_{\{s:s \in E, \rho(s,t) \leq \epsilon\}}|F_t(X_s) -F_t(X_t)| > \epsilon^2)  \leq 10kC\epsilon^{-2 +\frac{\beta}{\theta}}.
\end{equation}
Given our assumption that (\ref{eq3.12}) holds with $p=1$ and some $\beta \in (0,1]$, we take $\theta = \frac{\beta}{4}$, and hence (\ref{eq3.21}) implies we have the $L$-condition in (\ref{eq3.18}) with
$L= 10kC< \infty.$

Now we assume $\{Y(t): t \geq 0 \}$ is a process with stationary independent increments satisfying (\ref{eq3.12}) with $p \in (0,1)$ and some $\beta \in (0,1]$. Applying (\ref{eq3.16}) to (\ref{eq3.18}) we again have (\ref{eq3.19}), and as before
\begin{equation}\label{eq3.23}
A_{\epsilon} = \sup_{ t \in E} P(\sup_{\{s: s \in [t, (t+\epsilon^{\frac{1}{\theta}})\wedge T]\}}|Y_s -Y_t| > \frac{\epsilon^2}{2k}),
\end{equation}

Since $\{Y(t): t \in E \}$ is a process with stationary independent increments and cadlag sample paths, an application of Montgomery-Smith's maximal inequality in \cite{monty-comparison} implies
$$
A_{\epsilon} \leq 3 \sup_{ t \in E} P(|Y_{(t + \epsilon^{\frac{1}{\theta}})\wedge T} -Y_t| > \frac{\epsilon^2}{20k}).
$$
This maximal inequality is stated for sequences of i.i.d. random variables, but since $\{Y(t): t \in E \}$ is a process with stationary independent increments and cadlag sample paths, for any integer $n$ we can partition any subinterval I of E into $2^n$ equal subintervals and apply \cite{monty-comparison} to the partial sums formed from increments over each of these subintervals. One can add auxiliary i.i.d. increments to form a sequence, but that is not necessary as at the $n^{th}$each stage we need only work with  the partial sums of the $2^n$ increments of that stage. We then use \cite{monty-comparison} for an upper bound, and then pass via an increasing limit to what is needed, i.e. the desired upper bound is fixed, and hence is an upper bound for the limit.

Thus by Markov's inequality, and our assumption of (\ref{eq3.12}), we have
\begin{equation}\label{eq3.24}
A_{\epsilon}  \leq  3(20k\epsilon^{-2})^p \sup_{ t \in E} \E(|Y_{(t + \epsilon^{\frac{1}{\theta}})\wedge T} -Y_t|^p) \leq 3C (20k\epsilon^{-2})^p \epsilon^{\frac{\beta}{\theta}}. 
\end{equation}
We also have
\begin{align*}
B_{\epsilon}&\leq  \sup_{t \in E}P(|Y_{ (t-\epsilon^{\frac{1}{\theta}})\vee 0 } -Y_t| > \frac{\epsilon^2}{4k})\\
&\quad+ \sup_{t \in E}P(\sup_{\{s: s \in [ (t-\epsilon^{\frac{1}{\theta}})\vee 0, t ]\}}|Y_s - Y_{ (t-\epsilon^{\frac{1}{\theta}})\vee 0 }| > \frac{\epsilon^2}{4k}),
\end{align*}
and using Montgomery-Smith's maximal inequality again we have
$$
B_{\epsilon}\leq 4 \sup_{t \in E}P(|Y_{ (t-\epsilon^{\frac{1}{\theta}})\vee 0 } -Y_t| > \frac{\epsilon^2}{40k}).
$$
Thus by Markov's inequality and (\ref{eq3.12})
\begin{equation}\label{eq3.25}
B_{\epsilon}\leq 4(40k\epsilon^{-2})^p \sup_{ t \in E} \E(|Y_{(t + \epsilon^{\frac{1}{\theta}})\wedge T} -Y_t|^p) \leq 4C (40k\epsilon^{-2})^p \epsilon^{\frac{\beta}{\theta}}.  
\end{equation}
Combining (\ref{eq3.19}),(\ref{eq3.24}), and (\ref{eq3.25}) we have 
$$
\sup_ {t \in E} P(\sup_{\{s:s \in E, \rho(s,t) \leq \epsilon\}}|F_t(X_s) -F_t(X_t)| > \epsilon^2)  \leq 7C (40k)^p \epsilon^{\frac{\beta}{\theta}-2p},
$$
and hence the $L$-condition holds with $L =7C (40k)^p$ provided $\theta = \frac{\beta}{2+2p}$.
\end{proof}

\begin{cor}\label{cor3.3}
Let $E=[0,T]$, and assume $\{Y(t): t \geq 0 \}$ is a strictly stable process of index $r \in (0,2]$ with stationary independent increments, cadlaq sample paths on $[0,\infty)$, and such that  $P(Y(0)=0)=1$.   Let $X(t)= Z + Y(t), t \geq 0$, where  $Z$ is a random variable independent of $\{Y(t): t \geq 0 \}$ and having  density $g(\cdot)$ on $\mathbb{R}$ satisfying (\ref{eq3.15}).
Then, the empirical process built from i.i.d. copies of  $\{X(t)\colon\   t \in E \}$ satisfies the CLT over $\mathcal{C}$. Moreover, except for the degenerate cases when $r=1$ and $\{Y(t): t \geq 0 \}$ is pure drift, or $Y(t)$ is degenerate at zero for all $t \in E$, the empirical CLT over $\mathcal{C}$ fails for these $\{Y(t): t \in E \}$. 
\end{cor}

\begin{proof}
The assertions about the CLT holding are immediate consequences of  
Proposition \ref{prop3.3} once we check that $\{Y(t): t \geq 0 \}$ satisfies (\ref{eq3.12}). This follows from the comments immediately following (\ref{eq3.12}), and hence this part of the proof is established. 

To show that the CLT  fails for the strictly stable stationary independent increment processes specified follows from an application of the Hewitt-Savage zero-one law, and the scaling property of such processes. The case $r=2$ was previously established in \cite{kkz} using a law of the iterated logarithm argument, which also applied to all fractional Brownian motions.  Here we exploit the stationary independent increments of the processes to obtain a proof, and there are two other facts we need to emphasize at this point. The first is that $\{Y(t): t \geq 0 \}$  strictly stable of index $r \in (0,2]$ implies $\{S(t)= t^{\frac{2}{r}}Y(\frac{1}{t}): t \geq 0\}$ is also strictly stable of index $r$. The second is that the process $\{Y(t): t \geq 0 \}$ being non-degenerate and strictly stable, implies the random variables $Y(t)$ and $S(t)$ have probability densities for all $t>0$.

Now fix $n \geq 1$ and let $Y_1,\cdots,Y_n$ be independent copies of $Y$. Let $Q$ denote the rational numbers. Then, setting $S_j(t)= t^{\frac{2}{r}}Y(\frac{1}{t}), t > 0, j = 1,\cdots,n$, we have
\begin{align}\label{eq3.26}
P(\rm{card}\{Y_1(t),\cdots,Y_n(t)\}&= \rm{card}\{S_1(t),\cdots,S_n(t)\}=n\\
&\quad \text{for all }t\in Q\cap(0,\infty))=1.\notag
\end{align}
Also, let $\mathcal{C}_Q$ denote the countable subclass of $\mathcal{C}$ given by $\mathcal{C}_Q= \{C_{t,y} \in \mathcal{C}\colon$ $t, y \in Q\}$. Then, as in the proof of Lemma 7 in \cite{kkz}, to show the empirical CLT fails for $\{Y(t): t \in E\}$ it suffices to show that 
\begin{equation}\label{eq3.27}
P(\Delta^{\mathcal{C}_Q}(Y_1,\cdots,Y_n)= 2^n)=1,
\end{equation}
where $\Delta^{\mathcal{C}_Q}(Y_1,\cdots,Y_n)=\rm~{card}\{C\cap \{Y_1,\cdots,Y_n\}: C \in \mathcal{C}_Q\}$. Hence it suffices to show for every $r, 0\le r \leq n,$ and $\{i_1,\cdots, i_r\} \subseteq \{1,\cdots,n\}$ that
\begin{equation}\label{eq3.28}
P(\{Y_{i_1},\cdots,Y_{i_r}\} \in \Delta^{\mathcal{C}_Q}(Y_1,\cdots,Y_n))=1.
\end{equation}
Next define for every permutation $\pi=(j_1,\cdots,j_n)$ of $\{1,\cdots,n\}$ the event
\begin{equation}\label{eq3.29}
E_{\pi} =\{ \omega: Y_{j_1}(\frac{1}{k},\omega)< \cdots < Y_{j_n}(\frac{1}{k},\omega)~i.o. ~\rm {in}~k \geq 1\},
\end{equation}
and observe that
\begin{equation}\label{eq3.30}
E_{\pi} =\{ \omega: S_{j_1}(k,\omega)< \cdots < S_{j_n}(k,\omega)~i.o. ~\rm {in}~k \geq 1\}.
\end{equation}
Since there are only finitely many permutations and (\ref{eq3.26}) holds, $P(E_{\pi})>0$ for some $\pi$. Therefore, (\ref{eq3.30}) and the Hewitt-Savage zero-one law implies $P(E_{\pi})=1$. Moreover, since the processes $Y_1,\cdots,Y_n$ are i.i.d. it follows that $P(E_{\pi})=1$ for all permutations $\pi$ of $\{1,\cdots,n\}$.

Returning to (\ref{eq3.28}) we take any permutation $\pi=(j_1,\cdots,j_n)$ with $i_1=j_1, \cdots, i_r=j_r$. Then for each $\omega$ and $k$ such that (\ref{eq3.29}) holds we have a rational number $q(\omega,k)$ such that
$$
Y_{j_1}(\frac{1}{k},\omega)< \cdots < Y_{j_r}(\frac{1}{k},\omega)<q(\omega, k) < Y_{j_{r+1}}(\frac{1}{k},\omega)< \cdots < Y_{j_n}(\frac{1}{k},\omega),
$$
and hence 
$$
\{Y_{i_1}(\cdot,\omega), \cdots ,Y_{i_r}(\cdot,\omega)\} = C_{\frac{1}{k}, q(k,\omega)} \cap\{Y_1(\cdot,\omega),\cdots, Y_n(\cdot,\omega)\}.
$$
Since $P(E_{\pi})=1$, we therefore have (\ref{eq3.28}), which completes the proof.
\end{proof}

\section{Applications to empirical quantile process CLTs}\label{quantileclts}

The typical empirical quantile CLT of this section starts with a base process $\{Y_t: t \in E\}$, and as before we define $X_t = Y_t + Z, t \in E$, where $Z$ is independent of $\{Y_t: t \in E\}$ and $Z$ has density $g(\cdot)$ on $ \mathbb{R}$. For the empirical process CLT's over $\mathcal{C}$ established in the previous section, we assumed $g(\cdot)$ was uniformly bounded on $\mathbb{R}$, or in $L_a(\mathbb{R})$ for some $a>1$. In order to prove our empirical quantile results, we assume a bit more about $g(\cdot)$, but these assumptions are not unusual, even for real-valued quantile CLT's. Moreover, keeping in mind  possible application to a diverse collection of base processes, we have chosen to put the assumptions we require on $g(\cdot)$, but the reader should keep in mind that if the distributions of $Y_t, t \in E,$ have densities with similar properties, then we could assume less about $g(\cdot)$. This is easily seen from the proofs, and basic facts about convolutions, 
 and hence are left for the reader to implement should the occasion arise.

Throughout we assume enough that the input process $\{X_t:t \in E\}$ satisfies the empirical CLT over $\mathcal{C}$ with centered Radon Gaussian limit on $\ell_{\infty}(E \times \mathbb{R})$ given by $\{G(t,x): t \in E, x \in \mathbb{R}\},$ where $G(\cdot,\cdot)$ is sample bounded on $E \times \mathbb{R}$, and uniformly continuous with respect to its $L_2$-distance there. Of course, as before a typical point $(t,x) \in E\times \mathbb{R}$ has been identified with $C_{t,x}$. Our empirical quantile CLT's in this setting will then be of two types, and in these results $I$ will always be a closed subinterval of $(0,1)$. The first is that the quantile processes
\begin{equation}\label{quantile-1}
\{\sqrt n ( \tau_{\alpha}^n(t)- \tau_{\alpha}(t)) f(t, \tau_{\alpha}(t))\colon\  n \geq 1\} 
\end{equation}
satisfy the CLT in $\ell_{\infty}(E \times I)$ with Gaussian limit process 
\begin{equation}\label{quantile-2}
\{G(t,\tau_{\alpha}(t))\colon\  (t,\alpha) \in E \times I\},
\end{equation}
and the second asserts that the quantile processes
\begin{equation}\label{quantile-3}
\{\sqrt n (\tau_{\alpha}^n(t)- \tau_{\alpha}(t)) \colon\  n \geq 1\}
\end{equation}
satisfy the CLT in $\ell_{\infty}(E \times I)$ with Gaussian limit process 
\begin{equation}\label{quantile-4}
\left\{\frac{G(t,\tau_{\alpha}(t))}{f(t, \tau_{\alpha}(t))}\colon\  (t,\alpha) \in E \times I \right\}. 
\end{equation}

\begin{thm}\label{quantilethm}
Assume that one of (i--iii) hold:

(i) $\{Y_t: t \in E \}$ is a centered sample continuous Gaussian process on a compact subset $E$ of $[0,T]^d$ satisfying (\ref{euclidean-eq}).

(ii) $E=[0,T]$  and $\{Y(t): t \geq 0 \}$ is a stochastic process with cadlag sample paths on $[0,\infty)$ such that $P(Y(0)=0)=1$. In addition, 
 $\{Y(t)\colon\   t \in E \}$ is a martingale whose $L_1$-increments are Lip-$\beta$ for some $\beta \in (0,1]$ , or a stationary independent increments process satisfying (\ref{eq3.12}) for some $p \in (0,1)$ and $\beta \in (0, 1]$. 
 
 (iii) $E=[0,T]$ and $\{Y_t: t \in E \}$ is a compound Poisson process built from the i.i.d random variables $\{Y_k: k \geq 1\}$ having no mass at zero and Poisson process $\{N(t): t \ge 0\}$ with parameter $\lambda \in (0,\infty)$ as in Proposition 3.2.

\noindent In addition, assume $X_t = Y_t + Z$, where $Z$ is independent of  $\{Y_t\colon\  t \in E\}$, and $Z$ has a strictly positive, uniformly bounded, uniformly continuous density function $g$ on $\mathbb{R}$. If $\{Y_t: t \in E\}$ satisfies (i), (ii), or (iii), $I$ is a closed subinterval of $(0,1)$, and we also assume that
\begin{equation}\label{Y-prob-bound}
\lim_{b \rightarrow \infty}\sup_{t \in E} P(|Y_t| \ge b)=0,
\end{equation}
then the quantile processes of (\ref{quantile-1}) and (\ref{quantile-3}) built from the input process $\{X_t: t \in E\}$ satisfy the empirical quantile CLT with corresponding Gaussian limit as in (\ref{quantile-2}) and(\ref{quantile-4}).
\end{thm} 

\begin{rem}\label{rem3.2}
It is easy to see at this point that the results of section three allow us to apply Theorem 4.1 to obtain empirical quantile results of both types for fractional Brownian motions, the Brownian sheet, strictly stable stationary independent increment processes, martingales, and compound Poisson processes. The precise corollaries are easy to formulate, and hence are not included.
\end{rem}

\begin{proof}
If the base process $\{Y_t: t \in E\}$ satisfies (i), (ii), or (iii), and $g$ is uniformly bounded on $\mathbb{R}$, then Propositions 3.1-3.3 and Remark 3 following Proposition 3.2 imply that the resulting input process $\{X_t:t \in E\}$ satisfies the empirical CLT over $\mathcal{C}$ with centered Gaussian limit given by $\{G(t,x): t \in E, x \in \mathbb{R}\},$ where $G(\cdot,\cdot)$ is sample bounded on $E \times \mathbb{R}$, uniformly continuous with respect to its $L_2$-distance there, and has Radon support in $\ell_{\infty}(E \times \mathbb{R})$. Furthermore, if $H_t(x), t \in E,$ is the distribution function of $Y_t,$ then $X_t$ has probability density function
$$
f(t,x)= \int_{\mathbb{R}}g(x-v)dH_t(v), t \in E.
$$
Hence if $g$, the density of $Z$, is strictly positive, uniformly bounded, and uniformly continuous on $\mathbb{R}$, then it is easy to check that each of the densities $f(t,\cdot), t \in E,$ have the same properties. In particular,
we have $ \lim_{\delta\to 0}\sup_{t \in E}$ $\sup_{|u-v|\le\delta}|f(t,u) -f(t,v)| =0$, which is (\ref{unif-cont-densities}).
Hence Corollary 2.1 immediately implies the empirical quantile processes satisfy the quantile CLT's with Gaussian limit as indicated in (\ref{quantile-2}) and (\ref{quantile-4})
provided we show (\ref{inf-eq}) holds. That is, it remains to verify that the densities $f(t,\cdot), t \in E,$ of the input process $\{X_t: t \in E\}$ satisfy 
\begin{equation}\label{inf-alpha-I}
\inf_{t \in E,\alpha \in I} f(t,\tau_{\alpha}(t)) =c_I>0
\end{equation}
for every closed interval $I$ in $(0,1)$. 

Now (\ref{inf-alpha-I}) holds if we show that for any closed subinterval $I$ of $(0,1)$ and all $a>0$ that
\begin{equation}\label{inf-alpha-I-2}
\inf_{t \in E, |x| \leq a}f(t,x)=c_a>0 \text{ and } \sup_{t \in E, \alpha \in I} |\tau_{\alpha}(t)| < \infty. 
\end{equation}

First we show the left expression in (\ref{inf-alpha-I-2}) holds, so take $a>0$.
Then, for every $b>0$
\begin{align*}
\inf_{t \in E, |x| \leq a} f(t,x) &= \inf_{t \in E, |x| \leq a} \int_{\mathbb{R}} g(x-v)dH_t(v)\\
& \geq \inf_{t \in E} \int_{\mathbb{R}} \inf_{|x| \leq a}g(x-v)dH_t(v)\\
&\geq \inf_{|u| \leq a + b}g(u)  \inf_{t \in E}\int_{-b}^{b}dH_t(v),
\end{align*}
and, since  $g$ satisfies (\ref{Y-prob-bound}), there exists $b_0>0$ sufficiently large that
$$
 \inf_{t \in E} \int_{-b_0}^{b_0} dH_t(v) \geq  \frac{1}{2}.
$$
Therefore,  we have 
$$
\inf_{t \in E, |x| \leq a} f(t,x) \geq \frac{1}{2} \inf_{|u| \leq a + b_0}g(u) \equiv c_a>0.
$$

Now we turn to the second term in (\ref{inf-alpha-I-2}). Since $I$ is a closed interval of $(0,1)$ there is a $\theta \in (0, \frac{1}{2})$ such that $I \subset (\theta, 1-\theta)$ and 
\begin{align*}
\sup_{t \in [0,T]} P(|Y(t) +Z| \geq a) &\leq \sup_{t \in [0,T]} P\left(|Y(t)| \geq \frac{a}{2}\right) +P\left(|Z| \geq \frac{a}{2}\right)\\
&  \leq \frac{\theta}{2},
\end{align*}
where the second inequality follows from (\ref{Y-prob-bound}) by taking $a>0$ sufficiently large. Hence for each $t \in [0,T], \alpha \in I$ we have
$\tau_{\alpha}(t) \in [-a,a]$ and the right term of (\ref{inf-alpha-I-2}) holds. Thus (\ref{inf-alpha-I}) holds, and the theorem is proven.
\end{proof}

\section{Additional quantile process CLTs with stable and Gaussian inputs}\label{sec5}

Let $E=[0,T]$, and assume $\{X(t): t \geq 0 \}$ is a symmetric stable process of index $r \in (0,2]$ with stationary independent increments, cadlaq sample paths on $[0,\infty)$, and such that  $P(X(0)=0)=1$.   Then, except for degenerate cases, Corollary \ref{cor3.3} implies the empirical process built from i.i.d. copies of  $\{X(t)\colon\   t \in E \}$ fails the CLT over $\mathcal{C}$, but it holds
for i.i.d. copies of  $\{X(t)\colon\  a \le t \le T \}$ over $\mathcal{C}_{[a,T]} \equiv  \{C_{t,x}: a \le t \le T, x \in \mathbb{R}\}$, provided $0<a<T$. Moreover, with $E=[a,T]$ and $I$ a closed subinterval of $(0,1)$, Theorem \ref{quantilethm} then implies that the empirical quantile processes given in (\ref{quantile-1}) and (\ref{quantile-3}), satisfy the CLT with limiting Gaussian processes as in (\ref{quantile-2}) and (\ref{quantile-4}), respectively.

As can be seen from the proofs,  the difference in the results for $X$ indexed by $[0,T]$ versus $[a,T]$ seems in large part due to the fact that the densities $f(t,\cdot)$ are not uniformly bounded on $\mathbb{R}$ as $ t \downarrow 0$, and that $X(0)$ is degenerate at zero when $t=0$.  However, this is not the complete story, since in this section we will prove 
that the empirical quantile processes of (\ref{quantile-3}) with input process $X$ on $E=[0,T]$ satisfy the CLT of (\ref{quantile-4}) provided $I$ is a closed subinterval of $(0,1)$.
As mentioned in the introduction, this extends the result of J. Swanson in (\cite{swanson-scaled-median}) when $I=\{\frac{1}{2}\}$ and in (\cite{swanson-fluctuations})
for other fixed $\alpha \in (0,1)$. 

To prove our result we need a number of lemmas. The first shows that if the input process is scalable, then certain information on the empirical quantile process on an interval, say, $[1,2]$, yields information on an interval $[0,\delta]$. We phrase this in slightly more general terms, but ultimately it will be applied to quantile processes. 

Let $\{W(t,\alpha):t\geq 0, \alpha\in (0,1)\}$ be a stochastic processes which is scalable in $t$, i.e. for some constant $p \in (0,\infty)$ the process 
\[\{W(c t,\alpha):t\geq 0, \alpha\in (0,1)\} \text{ \rm and } \{c^pW(t,\alpha):t\geq 0, \alpha\in (0,1)\}
\] 
have the same law for all $n\geq 1$. Also assume $P(W(0,\alpha)=0)=1$ for all $\alpha \in (0,1)$.
Let $\mathbb{Q}$ denote the rational numbers, $J=[1,2]$, $A=[1-\alpha^{*},\alpha^{*}], \frac{1}{2} <\alpha^{*} <1$, and for a subset $B$ of $\mathbb{R}$ we define $B_{\mathbb{Q}}= B \cap \mathbb{Q}$. 

\begin{lem}\label{simplified homogeneous} Let $W$ be p-scalable. Fix $0<\delta\in \mathbb{Q}$. For $q>0$ 
\begin{equation}
\E(\sup_{ u \in (0,\delta]_{\mathbb{Q}}, \alpha\in A_{\mathbb{Q}}}|W(u, \alpha)|^q])\le \dfrac{\delta^{pq}}{1-2^{-pq}}\E(\sup_{ u \in J_{\mathbb{Q}}, \alpha\in A_{\mathbb{Q}}}|W(u, \alpha)|^q]).
\end{equation}
\end{lem}

\begin{proof}
\begin{align*}
\E\sup_{ s \in (0,\delta]_{\mathbb{Q}}, \alpha\in A_{\mathbb{Q}}}|W(s, \alpha)|^q]&=\E \sup_{j\ge 1}\sup_{s\in (2^{-j}\delta, 2^{-(j-1)}\delta]_{\mathbb{Q}}, \alpha\in A_{\mathbb{Q}}}|W(s,\alpha)|^{q}\\
&\le\sum_{j=1}^{\infty}\ \E \sup_{s\in (2^{-j}\delta, 2^{-(j-1)}\delta]_{\mathbb{Q}}, \alpha\in A_{\mathbb{Q}}}|W(s,\alpha)|^{q}\\
&=\sum_{j=1}^{\infty}\ \E \sup_{s\in (1,2]_{\mathbb{Q}}, \alpha\in A_{\mathbb{Q}}}|W(2^{-j}\delta s,\alpha)|^{q}\\
&=\sum_{j=1}^{\infty} (2^{-j}\delta)^{pq}\ \E\sup_{s\in (1,2]_{\mathbb{Q}}, \alpha\in A_{\mathbb{Q}}}|W(s,\alpha)|^{q}\\
&=\frac{\delta^{pq}}{1-2^{-pq}}\E(\sup_{ u \in J_{\mathbb{Q}}, \alpha\in A_{\mathbb{Q}}}|W(u, \alpha)|^q]).\qquad \qed
\end{align*}
\renewcommand{\qed}{}\end{proof}

\begin{rem}Shortly we will apply this to the sequence of empirical quantile processes. That is, we apply this lemma to each of  $W_{n}(t,\alpha):=\sqrt{n}\bigl(F^{-1}_{n,t}(\alpha)-F^{-1}_{t}(\alpha)\bigr)$, where $t\ge 0, \alpha\in(0,1)$. Since the bounds obtained in  Lemma \ref{simplified homogeneous} depend only on the scalability constant, $c$, all of our estimates will be uniform in $n$. 
\end{rem}

We now prove that when the input process is a stationary, independent increment process with symmetric $p$-stable distribution ($0<p<2$), the empirical quantile processes uniformly satisfy the hypothesis of Lemma \ref{simplified homogeneous}. For this purpose the next  lemma is useful.

\begin{lem}\label{quantile comparison} Let $X$ be an arbitrary random variable. If $q_{\alpha}(X)$ denotes any $\alpha$-quantile for $X$, then $-q_{1-\alpha}(-X)$ is also an $\alpha$-quantile for $X$. 
\end{lem}

\begin{proof} 
\begin{align*}
&P(X\ge -q_{1-\alpha}(-X))=P(-X\le q_{1-\alpha}(-X))\ge 1-\alpha\\
\intertext{and}
&P(X\le -q_{1-\alpha}(-X))=P(-X\ge q_{1-\alpha}(-X))\ge \alpha. \qquad \qed
\end{align*}
\renewcommand{\qed}{}\end{proof}

\begin{thm}\label{tail bounds}
Let $\{X(t)\colon t\ge 0\}$ be a symmetric $r$-stable process with stationary, independent increments, and such that $P(X(0)=0)=1$. Then, the centered empirical quantile process built from i.i.d copies of $\{X(t): t \ge 0\}$ satisfies the hypothesis of Lemma \ref{simplified homogeneous}, i.e., there exists a positive integer $n_0$ such that
$$
\sup_{n \ge n_0}\E[\sup_{t\in J_{\mathbb{Q}}, \alpha\in A_{\mathbb{Q}}} \sqrt n |F_{n,t}^{-1}(\alpha)-F_{t}^{-1}(\alpha)|] <\infty,
$$
and hence for every  $\epsilon>0$ there exists $\delta>0$ such that
$$
\sup_{n \geq n_0} P(\sup_{t \in [0, \delta]_ Q, \alpha \in A_{\mathbb{Q}}} \sqrt n |F_{n,t}^{-1}(\alpha)-F_{t}^{-1}(\alpha)|> \epsilon) \leq \epsilon.
$$
\end{thm}

\begin{rem}
From (\ref{final minmax-1}), the $n_0$ of the theorem can be taken to be 
$$
n_0= \inf \{n \ge 1: 2^{-(r \lfloor n(1-\alpha^{*})\rfloor -2)}(\lambda_r \sqrt n)^2 \le 1\},
$$ 
where $\lambda_{r}^{r}:=\dfrac{2^{r}e c_{r}}{1-\alpha^{*}}$, and $c_r$ depends on the tail behavior of $||X||_J$ as given prior to (\ref{final minmax}). In addition, due to the scaling of our input process, $F_t^{-1}(\alpha) = t^{\frac{1}{r}}F_1^{-1}(\alpha)$, and since $\alpha \in A$ implies $|F_1^{-1}(\alpha)| \le F_1^{-1}(\alpha^{*})< \infty$, we have for $0<q<r$ that
$$
\sup_{1 \le n < n_0} \E[\sup_{t\in J_{\mathbb{Q}}, \alpha\in A_{\mathbb{Q}}} \sqrt n |F_{n,t}^{-1}(\alpha)-F_{t}^{-1}(\alpha)|^q] <\infty,
$$
i.e.,  $\max_{1 \le n \le n_0} \sup_{t \in J_{\mathbb{Q}},\alpha \in A_{\mathbb{Q}}}|F_{n,t}^{-1}(\alpha)| \le \max_{1 \le j \le n_0} \sup_{0 \le t \le T} |X_j(t)|<\infty,$
and for $0< q <r$ the right term has a $q^{th}$ moment. Thus
$$
\sup_{n \ge 1} \E[\sup_{t\in J_{\mathbb{Q}}, \alpha\in A_{\mathbb{Q}}} \sqrt n |F_{n,t}^{-1}(\alpha)-F_{t}^{-1}(\alpha)|^q] <\infty
$$
for $0<q<r$, and the conclusion of Theorem \ref{tail bounds} also holds for all $n\ge 1$ via an application of Lemma \ref{simplified homogeneous} with $n_0=1$.
\end{rem}

\begin{proof} First we note that the scaling property of the iid symmetric stable processes, $\{X_{j}(t)\}$ immediately implies scalability with the same constant for all the processes 
\[\{F_{n,t}^{-1}(\alpha)-F_{t}^{-1}(\alpha)\colon t\in [0,\infty), \alpha\in (0,1)\}.
\] 

\n We'll obtain bounds on 
\begin{equation}P(\sqrt{n}\sup_{t\in J_{\mathbb{Q}},\alpha\in A_{\mathbb{Q}}}|F_{n,t}^{-1}(\alpha)-F_{t}^{-1}(\alpha)|>u).
\end{equation}
strong enough to yield an $n_{0}$ for which 
\[\sup_{n \ge n_0} \E[\sup_{t\in J_{\mathbb{Q}}, \alpha\in A_{\mathbb{Q}}} \sqrt n |F_{n,t}^{-1}(\alpha)-F_{t}^{-1}(\alpha)|] <\infty,
\]
At this point we can apply Lemma \ref{simplified homogeneous} to obtain the bound
\begin{align}
&\E(\sup_{ u \in (0,\delta]_{\mathbb{Q}}, \alpha\in A_{\mathbb{Q}}}|F_{n,t}^{-1}(\alpha)-F_{t}^{-1}(\alpha)|^q])\le \dfrac{\delta^{pq}}{1-2^{-pq}}\\
&\E(\sup_{ u \in J_{\mathbb{Q}}, \alpha\in A_{\mathbb{Q}}}|F_{n,t}^{-1}(\alpha)-F_{t}^{-1}(\alpha)|^q]).\notag
\end{align}
An application of Chebyschev's inequality will then yield the Theorem. 
\par

We break the proof into two parts. The first part covers the case when we have a lower bound on the densities of $X_{t}$ for $t\in J=[1,2]$. In the second part we take care of the remaining case.
Now for $u\ge 0$ we have
\begin{enumerate}
\item[(i)] 
\begin{align*}
&P(\sqrt{n}\sup_{t\in J_{\mathbb{Q}}, \alpha\in A_{\mathbb{Q}}}\bigl(F_{n,t}^{-1}(\alpha)-F_{t}^{-1}(\alpha)\bigr)>u)\\
=~ &P(\exists t \in J_{\mathbb{Q}}, \alpha \in A_{\mathbb{Q}}, \sum_{j=1}^{n}I_{X_{j}(t)>\dfrac{u}{\sqrt{n}}+F_{t}^{-1}(\alpha)}\ge n(1-\alpha))\\
=~ &P(\exists t \in J_{\mathbb{Q}},\alpha \in A_{\mathbb{Q}},  \sum_{j=1}^{n}\bigl(I_{X_{j}(t)>{u}/{\sqrt{n}}+F_{t}^{-1}(\alpha)}-P(X(t)\\
&\quad >{u}/{\sqrt{n}}+F_{t}^{-1}(\alpha))\bigr)
\ge n\bigl[(1-\alpha) - P(X(t)>\dfrac{u}{\sqrt{n}}+F_{t}^{-1}(\alpha))\bigr])\\
=~ &P(\exists t \in J_{\mathbb{Q}},\alpha \in A_{\mathbb{Q}}, \sum_{j=1}^{n}\bigl(I_{X_{j}(t) >{u}{\sqrt{n}}+F_{t}^{-1}(\alpha)}-P(X(t)\\
&\quad >\dfrac{u}{\sqrt{n}}+F_{t}^{-1}(\alpha))\bigr)
\ge nP(F_{t}^{-1}(\alpha)<X(t)\le \dfrac{u}{\sqrt{n}}+F_{t}^{-1}(\alpha))).
\end{align*}

Also, since the density, $f_{1}$, of $X(1)$, is symmetric about $0$ and unimodal (\cite{yamazato-unimodal}), it is decreasing away from the origin. Hence, using $0< t \le 2$, 
\begin{align*}P&(F_{t}^{-1}(\alpha)<X(t)\le \dfrac{u}{\sqrt{n}}+F_{t}^{-1}(\alpha))=\int_{t^{-1/r}F_{t}^{-1}(\alpha)}^{t^{-1/r}[F_{t}^{-1}(\alpha)+\frac{u}{\sqrt{n}}]}f_{1}(x)\, dx\\
&\ge \bigl(\inf_{F_{1}^{-1}(\alpha)\le x\le [F_{1}^{-1}(\alpha)+\frac{t^{{-1}/{r}}u}{\sqrt{n}}]}f_{1}(x)\bigr)\dfrac{u}{2^{\frac{1}{r}}\sqrt{n}}\\
&\ge f_{1}(F_{1}^{-1}(\alpha^{*})+\frac{u}{\sqrt{n}})\dfrac{u}{2^{1/r}\sqrt{n}}
\end{align*}
So, if  $0 \le \dfrac{u}{\sqrt{n}}\le C$, since the density, $f_{1}$, is decreasing away from the origin,we have the inequality
\begin{align*}
&P (\sqrt{n}\sup_{t\in J_{\mathbb{Q}}, \alpha\in A_{\mathbb{Q}}}\bigl(F_{n,t}^{-1}(\alpha)-F_{t}^{-1}(\alpha)\bigr)>u)\\
\le &P(\dfrac1{\sqrt{n}}\|\sum_{j=1}^{n}\bigl(I_{X_{j}(t)>y}-P(X(t)>y)\bigr)\|_{J_{\mathbb{Q}}\times \mathbb{R}}\\
\ge  &f_{1}(F_{1}^{-1}(\alpha^{*})+C)u/2^{1/r}).
\end{align*}
Thus, for $t \in J_{\mathbb{Q}}$ fixed, the continuity of $P(X(t) >y)$ in $y$ and the right continuity of $I_{X_j(t) >y}$ in $y$ for $1 \le j \le n$  implies
\begin{align*}
&\dfrac1{\sqrt{n}}\|\sum_{j=1}^{n}\bigl(I_{X_{j}(t)>y}-P(X(t)>y)\bigr)\|_{J_{\mathbb{Q}}\times \mathbb{R}}\\
=~ &\dfrac1{\sqrt{n}}\|\sum_{j=1}^{n}\bigl(I_{X_{j}(t)>y}-P(X(t)>y)\bigr)\|_{J_{\mathbb{Q}}\times \mathbb{Q}}
\end{align*}
with probability one.

Hence, if $D:= 2^{-1/r}f_{1}(F_{1}^{-1}(\alpha^{*})+C)$, we have for $0 \le \dfrac{u}{\sqrt{n}}\le C$, 
\begin{align}\label{emp upper bound}P&(\sqrt{n}\sup_{t\in J_{\mathbb{Q}}, \alpha\in A_{\mathbb{Q}}}\bigl(F_{n,t}^{-1}(\alpha)-F_{t}^{-1}(\alpha)\bigr)>u)\notag\\
&\le P(\dfrac1{\sqrt{n}}\|\sum_{j=1}^{n}\bigl(I_{X_{j}(t)>y}-P(X(t)>y)\bigr)\|_{J_{\mathbb{Q}}\times \mathbb{Q}}\ge  Du)
\end{align}
Since the summands, $I_{X_{j}(t)>y}-P(X(t)>y)$ are bounded by $1$, and the CLT  over $\mathcal{C}$ implies stochastic boundedness of the normalized norm in (\ref{emp upper bound}), we can use a result of  Hoffman-J\o rgensen, see pp. 164-5 of  \cite{hoff}, 
to obtain for any $q>0$, 
\[B_{q}:=\sup_{n}\E \dfrac1{\sqrt{n}}\|\sum_{j=1}^{n}\bigl(I_{X_{j}(t)>y}-P(X(t)>y)\bigr)\|_{J_{\mathbb{Q}}\times \mathbb{Q}}^{q}<\infty.
\]
Therefore, for $0 \le  u \le C/\sqrt n,$
\begin{align}\label{small values}P&(\sqrt{n}\sup_{t\in J_{\mathbb{Q}}, \alpha\in A_{\mathbb{Q}}}\bigl(F_{n,t}^{-1}(\alpha)-F_{t}^{-1}(\alpha)\bigr)>u)\le B_{q}\dfrac{1}{(Du)^{q}}.
\end{align}

\item[(ii)] Now we deal with the case $\frac{u}{\sqrt{n}}> C$. 
In the  computation below we don't use the particular form of the quantiles, $F_{n,t}^{-1}(\alpha), F_{t}^{-1}(\alpha)$, only the fact that they are quantiles. Note (\ref{other quantile}) below. 
Hence, by Lemma \ref{quantile comparison} 
\begin{align}
&P(\sqrt{n}\sup_{t\in J,\alpha\in A}|F_{n,t}^{-1}(\alpha)-F_{t}^{-1}(\alpha)|>u)\notag\\
\le~ &P(\sqrt{n}\sup_{t\in J_{\mathbb{Q}},\alpha\in A_{\mathbb{Q}}}\bigl(F_{n,t}^{-1}(\alpha)-F_{t}^{-1}(\alpha)\bigr)>u)\notag\\
+~ &P(\sqrt{n}\sup_{t\in J_{\mathbb{Q}},\alpha\in A_{\mathbb{Q}}}\bigl(F_{t}^{-1}(\alpha)-F_{n,t}^{-1}(\alpha)\bigr)>u)\notag\\
 =~ &P(\sqrt{n}\sup_{t\in J_{\mathbb{Q}},\alpha\in A_{\mathbb{Q}}}\bigl(F_{n,t}^{-1}(\alpha)-F_{t}^{-1}(\alpha)\bigr)>u)\notag\\
 +~ &P(\sqrt{n}\sup_{t\in J_{\mathbb{Q}},\alpha\in A_{\mathbb{Q}}}\bigl(-q_{t}(1-\alpha)+q_{n,t}(1-\alpha)\bigr)>u).\label{other quantile}
 \end{align} 
Thus, the second term can be treated the same as the first term. 
For the first term we have 
\begin{align}\label{minmax}P&(\sqrt{n}\sup_{t\in J_{\mathbb{Q}},\alpha\in A_{\mathbb{Q}}}\bigl(F_{n,t}^{-1}(\alpha)-F_{t}^{-1}(\alpha)\bigr)>u)\notag\\
&=P(\exists t\in J_{\mathbb{Q}},\alpha\in A_{\mathbb{Q}}, F_{n,t}^{-1}(\alpha)>\dfrac{u}{\sqrt{n}}+F_{t}^{-1}(\alpha))\notag\\
&= P(\exists t\in J_{\mathbb{Q}}, \alpha\in A_{\mathbb{Q}}, \exists I, \#I= \lfloor n(1-\alpha)\rfloor, X_{j}(t)\notag\\
&\quad >\dfrac{u}{\sqrt{n}}+F_{t}^{-1}(\alpha), \forall j\in I)\notag\\
&\le P(\exists t\in J_{\mathbb{Q}}, \alpha\in A_{\mathbb{Q}}, \exists I, \#I= \lfloor n(1-\alpha^{*}) \rfloor, X_{j}(t)\notag\\
&\quad >\dfrac{u}{\sqrt{n}}+F_{t}^{-1}(\alpha), \forall j\in I)\notag
\end{align}
and again by Lemma \ref{quantile comparison}, since  $F_{t}^{-1}(\alpha)\ge F_{t}^{-1}(1-\alpha^{*})~ \text{for all}~ \alpha \in A_{\mathbb{Q}}$,
\begin{align}
&\le \binom{n}{\lfloor n(1-\alpha^{*})\rfloor}P(\exists t\in J_{\mathbb{Q}}, X_{j}(t)>\dfrac{u}{\sqrt{n}}-F_{t}^{-1}(\alpha^{*}),\notag\\
&\quad  j=1,\ldots, \lfloor n(1-\alpha^{*})\rfloor) \notag\\
&\le\binom{n}{\lfloor n(1-\alpha^{*})\rfloor}P(\min_{j\le \lfloor n(1-\alpha^{*})\rfloor }\|X_{j}\|_{J_{\mathbb{Q}}}>\dfrac{u}{\sqrt{n}}-F_{1}^{-1}(\alpha^{*}))\notag\\
&\le\binom{n}{\lfloor n(1-\alpha^{*})\rfloor}\bigl[P(\|X\|_{J_{\mathbb{Q}}}>\dfrac{u}{\sqrt{n}}-F_{1}^{-1}(\alpha^{*}))\bigr]^{\lfloor n(1-\alpha^{*})\rfloor }\notag\\
&\le \bigl[\dfrac{e}{1-\alpha^{*}}P(\|X\|_{J_{\mathbb{Q}}}>\dfrac{u}{\sqrt{n}}-F_{1}^{-1}(\alpha^{*}))\bigr]^{\lfloor n(1-\alpha^{*})\rfloor}
\end{align}
Now, for our stable process it is known, see Proposition 5.6 of  \cite{led-tal-book}, that there exists a constant, $c_{r}$, such that 
\[P(\|X\|_{J_{\mathbb{Q}}}>v)\le c_{r} v^{-r}.
\]
Hence, by (\ref{minmax}), if $\dfrac{u}{\sqrt{n}}\ge C\ge 2F_{1}^{-1}(\alpha^{*})$, we have for $\lambda_{r}^{r}:=\dfrac{2^{r}e c_{r}}{1-\alpha^{*}}$, 
\begin{align}\label{final minmax}
&P(\sqrt{n}\sup_{t\in J_{\mathbb{Q}},\alpha\in A_{\mathbb{Q}}}\bigl(F_{n,t}^{-1}(\alpha)-F_{t}^{-1}(\alpha)\bigr)>u)\notag\\
&\le \bigl[\dfrac{e}{1-\alpha^{*}}P(\|X\|_{J_{\mathbb{Q}}}>\dfrac{u}{2\sqrt{n}})\bigr]^{\lfloor n(1-\alpha^{*})\rfloor}\notag\\
&\le\bigl[\dfrac{e}{1-\alpha^{*}}c_{r}(\dfrac{2\sqrt{n}}{u})^{r}\bigr]^{\lfloor n(1-\alpha^{*})\rfloor }=\bigl[\frac{\lambda_r \sqrt n}{u}\bigr]^{r \lfloor n(1-\alpha^{*})\rfloor }.
\end{align}

Therefore, taking
$u/\sqrt n \ge C \equiv 2\lambda_r \vee 2F_1^{-1}(\alpha^{*})$,  (\ref{final minmax}) implies
\begin{align}\label{final minmax-1}P&(\sqrt{n}\sup_{t\in J_{\mathbb{Q}},\alpha\in A_{\mathbb{Q}}}\bigl(F_{n,t}^{-1}(\alpha)-F_{t}^{-1}(\alpha)\bigr)>u)\le\bigl[\frac{\lambda_r \sqrt n}{u}\bigr]^{r \lfloor n(1-\alpha^{*})\rfloor }\notag\\
 &\phantom{*************} \le 
2^{-(r \lfloor n(1-\alpha^{*})\rfloor -2)}(\frac{\lambda_r \sqrt n}{u})^2,
\end{align}
and hence $n$ sufficiently large implies
\begin{align}\label{final minmax-2}P(\sqrt{n}\sup_{t\in J_{\mathbb{Q}},\alpha\in A_{\mathbb{Q}}}\bigl(F_{n,t}^{-1}(\alpha)-F_{t}^{-1}(\alpha)\bigr)>u)\le u^{-2}.
\end{align}

Since the same estimates apply to the second term in (\ref{other quantile}), we have by putting the two parts together that
\begin{align}
\E&[\sup_{t\in J_{\mathbb{Q}}, \alpha\in A_{\mathbb{Q}}} \sqrt n |F_{n,t}^{-1}(\alpha)-F_{t}^{-1}(\alpha)|]\notag\\
&\le 2 \int_{0}^{\infty}P(\sqrt{n}\sup_{t\in J_{\mathbb{Q}}, \alpha\in A_{\mathbb{Q}}}\bigl(F_{n,t}^{-1}(\alpha)-F_{t}^{-1}(\alpha)\bigr)>u)\, du\notag\\
&\le 2[1+\int_{0}^{C\sqrt{n}}P(\sqrt{n}\sup_{t\in J_{\mathbb{Q}}, \alpha\in A_{\mathbb{Q}}}\bigl(F_{n,t}^{-1}(\alpha)-F_{t}^{-1}(\alpha)\bigr)>u)\, du\notag\\
&\phantom{...........}+\int_{C\sqrt{n}}^{\infty}P(\sqrt{n}\sup_{t\in J_{\mathbb{Q}}, \alpha\in A_{\mathbb{Q}}}\bigl(F_{n,t}^{-1}(\alpha)-F_{t}^{-1}(\alpha)\bigr)>u)\, du]\notag\\
&\le 2[1+\dfrac{B_{2}}{D^{2}}\int_{1}^{C\sqrt{n}}\dfrac1{u^{2}}\, du +\int_{C\sqrt{n}}^{\infty}\dfrac1{u^{2}}\, du]  < \infty,
\end{align}
provided $n$ is sufficiently large, $B_2$ and $D$ are as in (\ref{small values}), and $C \equiv 2\lambda_r \vee 2F_1^{-1}(\alpha^{*})$. Thus the hypotheses in the Lemma \ref{simplified homogeneous} are uniformly satisfied. Hence the theorem is proved. $\hfill\qed$
\end{enumerate}
\renewcommand{\qed}{}\end{proof}

Our next lemma is important in that it allows us to switch back and forth between supremums over countable and uncountable parameter sets. For example, one consequence is that the conclusion of Theorem \ref{tail bounds} can be strengthened to hold for all $t \in J$ and $\alpha \in A$ provided we ask that the processes $\{X_j(t): t \ge 0\}, j \ge 1,$ have cadlag sample paths with probability one. The lemma is as follows. The sets $A$ and $A_{\mathbb{Q}}$ are as above, but $[0,T]_{\mathbb{Q}}$ also includes the point $T$, even if it is irrational.

\begin{lem}\label{ctble-unctble-sups}
Let $\{X(t)\colon t\ge 0\}$ be a symmetric $r$-stable process with stationary, independent increments, cadlag sample paths, and such that $P(X(0)=0)=1$. Let $0<T < \infty$, $\mathbb{Q}$ the rational numbers, and define $[0,T]_{\mathbb{Q}}= ([0,T] \cap \mathbb{Q})\cup \{T\}$. Then, the  empirical quantile process $\tau_{\alpha}^n(t)$ built from i.i.d copies of $\{X(t): t \ge 0\}$ with cadlag paths on a complete probability space has right continuous paths on $[0,T]$ with probability one, and is such that
\begin{equation}\label{ctble-unctble-0}
P(\sup_{t \in [0,T], \alpha \in A} |\tau_{\alpha}^n(t) - \tau_{\alpha}(t)| = \sup_{t \in [0,T]_{\mathbb{Q}}, \alpha \in A_{\mathbb{Q}}} |\tau_{\alpha}^n(t)) - \tau_{\alpha}(t)|)=1.
\end{equation} 
Moreover, for each $t \in [0,T]$ and $n \geq 1$ with probability one the empirical quantile process  $\tau_{\alpha}^n(t)$ is left continuous in $\alpha \in (0,1)$.
\end{lem}

\begin{proof} 
First we observe that for $t>0$ the distribution function $F(t,x)$ has strictly positive density
given by
\begin{equation}\label{ctble-unctble-1}
f(t,x)= (2\pi)^{-1}\int_{\mathbb{R}} \exp{\{-ct|u|^r\}}\cos(xu) du, c>0,
\end{equation}
and hence $F(t,x)$ is strictly increasing and  continuous in $ x \in \mathbb{R}$. Thus $\tau_{\alpha}(t) = F_t^{-1}(\alpha)$ is continuous in $\alpha \in (0,1)$, and by its definition, we also have $ \tau_{\alpha}^n(t)= F_{n,t}^{-1}(\alpha)$ left continuous in $\alpha \in (0,1)$. In particular, the claim following (\ref{ctble-unctble-0}) holds.
 Moreover, for every $\alpha \in (0,1)$, $P(\tau_{\alpha}(0) = \tau_{\alpha}^n(0)=0)=1$ since we are assuming $P(X_j(0)=0)=1, j\ge1.$ Therefore, we have 
\begin{equation}\label{ctble-unctble-2}
P(\sup_{t \in [0,T], \alpha \in A} |\tau_{\alpha}^n(t) - \tau_{\alpha}(t)| = \sup_{t \in [0,T], \alpha \in A_{\mathbb{Q}}} |\tau_{\alpha}^n(t)) - \tau_{\alpha}(t)|)=1.
\end{equation} 

We also have $\tau_{\alpha}(\cdot)$ continuous in $t$ on $[0,\infty)$, since scaling easily implies $\tau_{\alpha}(t) = t^{\frac{1}{r}}\tau_{\alpha}(1)$ for all $t \ge 0$. Thus (\ref{ctble-unctble-0})
follows from (\ref{ctble-unctble-2}) provided we show $\tau_{\alpha}^n(t)$ is right continuous on $[0,T)$ with probability one, i.e. we then would have
\begin{equation}\label{ctble-unctble-3}
P(\sup_{t \in [0,T], \alpha \in A_{\mathbb{Q}}} |\tau_{\alpha}^n(t) - \tau_{\alpha}(t)| = \sup_{t \in [0,T]_{\mathbb{Q}}, \alpha \in A_{\mathbb{Q}}} |\tau_{\alpha}^n(t)) - \tau_{\alpha}(t)|)=1.
\end{equation}

To verify the right continuity of $\tau_{\alpha}^n(\cdot)$, and (\ref{ctble-unctble-3}), we use the right continuity of the paths of the processes $X_1,\cdots,X_n$. We do this through the following observation.
That is, given real numbers $x_1,\cdots,x_n$, let $x_{1,1}, \cdots, x_{1,n}$ be an ordering of these numbers such that $x_{1,1} \leq \cdots \leq x_{1,n}$. In case there are no ties among  $\{x_1,\cdots,x_n\}$ this ordering is unique, and when there are ties, we choose the ordering based on the priority of the original index among the tied numbers. We then refer to $x_{1,1} \leq \cdots \leq x_{1,n}$ as the order statistics of $\{x_1,\cdots,x_n\}$. Now take pairs $(x_1,y_1), \cdots, (x_n,y_n)$ of real numbers such that $\sup_{1 \leq j \leq n}|x_j - y_j|  \leq \delta$. These are called the initial pairs of the two sets of $n$ numbers. We will now verify by induction that the corresponding order statistics formed from these sets also satisfy
\begin{equation}\label{ctble-unctble-4}
\sup_{1 \leq j \leq n}|x_{1,j} - y_{1,j}|\leq \delta. 
\end{equation}

The case $n=1$ is obvious, so assume the result holds for all sets with cardinality less than or equal to $ n-1$.  Then, assume that in the initial pairings, $x_{1,1}$ and $y_{1,1}$ are paired with, say $y_{1,k}$ and $x_{1,l}$, respectively. Hence we have $x_{1,1} \leq x_{1,l}, y_{1,1} \leq y_{1,k}$, 
$|x_{1,1} - y_{1,k}|\leq \delta$ and $|y_{1,1} - x_{1,l}|\leq \delta.$ If $x_{1,1} \leq y_{1,1}$, then from the above we have $x_{1,1} \leq y_{1,1} \leq y_{1,k}$, and hence $|x_{1,1} - y_{1,k}|\leq \delta$ implies
$|x_{1,1} - y_{1,1}|\leq \delta$. Similarly, if $x_{1,1} > y_{1,1}$, then from the above we have $y_{1,1} < x_{1,1} \leq x_{1,l}$, and hence $|y_{1,1} - x_{1,l}|\leq \delta$ implies
$|x_{1,1} - y_{1,1}|\leq \delta$.  We also have $|x_{1,l} - y_{1,k}|\leq \delta$. That is, if $x_{1,l} \leq y_{1,k}$, then we have $x_{1,1} \leq  x_{1,l} \leq y_{1,k}$ and hence $|x_{1,1} - y_{1,k}|\leq \delta$ implies
$|x_{1,l} - y_{1,k}|\leq \delta$. Similarly, if $x_{1,l} >y_{1,k}$, then we have $y_{1,1} \leq y_{1,k} < x_{1,l}$ and hence $|y_{1,1} - x_{1,l}|\leq \delta$ implies
$|x_{1,l} - y_{1,k}|\leq \delta$.

To finish the proof  of (\ref{ctble-unctble-4}) we apply the induction hypothesis to the set of $n-1$ pairs determining $x_{1,2} \leq \cdots \leq x_{1,n}$ and $y_{1,2} \leq \cdots \leq y_{1,n}$, 
with $(x_{1,l},y_{1,k})$ being   a possibly  new pair, and the remaining $n-2$ pairs are those originally given. Note that the induction hypothesis applies to these pairs, since we have shown $|x_{1,l}-y_{1,k}| \leq \delta.$ Thus (\ref{ctble-unctble-4}) holds.

To verify the right continuity of $\tau_{\alpha}^n(\cdot)$, and hence that (\ref{ctble-unctble-3}) holds, we note that since the i.i.d. processes $X_1,\cdots,X_n$ are cadlag on $[0,\infty)$ with probability one, there is a set $\Omega_1 \subseteq \Omega$ such that $P(\Omega_1)=1$ and for every $t \in [0,T), \epsilon>0,$ there is a $\delta=\delta(\omega,t,\epsilon,n)>0$ such that $\omega \in \Omega_1$ implies 
$$
\sup_{1 \leq j \leq n, t\le s \le (t+\delta) \wedge T} |X_j(s) - X_j(t)| \leq \epsilon.
$$
Therefore, (\ref{ctble-unctble-4}) implies the order statistics $X_{1,1}(s) \leq \cdots \leq X_{1,n}(s)$ and $X_{1,1}(t) \leq \cdots \leq X_{1,n}(t)$ obtained from $\{X_1(s),\cdots,X_n(s)\}$ and  $\{X_1(t),\cdots$, $X_n(t)\}$ are such that
\begin{equation}\label{ctble-unctble-5}
\sup_{1 \leq j \leq n, t\le s \le (t+\delta) \wedge T} |X_{1,j}(s) - X_{1,j}(t)| \leq \epsilon. 
\end{equation}
Since $\epsilon>0$ is arbitrary, we thus have that the order statistic processes $\{X_{1,j}(t): t \in [0,T)\}, j=1,\cdots,n,$ are right continuous on $\Omega_1$, and hence with probability one.

Now for $0<\alpha <1, n \geq 1, t \in [0,\infty)$ we have $\tau_{\alpha}^n(t) = \inf\{x: F_n(t,x) \geq \alpha\}$, and hence 
$$
\tau_{\alpha}^n(t) = X_{1,j(\alpha)}(t), 
$$
where $j(\alpha)= \min\{ k: 1 \leq k \leq n, k/n  \geq \alpha\}$ is independent of $t \in E$. 
Thus for all $\omega \in \Omega_1$ 
we have $\tau_{\alpha}^n(t)$ right continuous in $t \in [0,T)$. Hence the lemma is proven.
\end{proof}
 
\begin{cor}\label{prob-cont-zero}
Let $\{X(t)\colon t\ge 0\}$ be a symmetric $r$-stable process with stationary, independent increments, cadlag sample paths, and such that $P(X(0)=0)=1$. Also, assume the  empirical quantile processes $\tau_{\alpha}^n(t)$ are built from i.i.d copies of $\{X(t): t \ge 0\}$ with cadlag paths. Then, 
there exists an integer $n_0 \ge 1$ such that for every $\epsilon>0$ there is  a $\delta=\delta(\epsilon)>0$ satisfying
\begin{equation}\label{cont-zero}
\sup_{ n \ge n_0}P(\sup_{t \in [0,\delta], \alpha \in A} \sqrt n |\tau_{\alpha}^n(t) - \tau_{\alpha}(t)| >\epsilon)\le \epsilon. 
\end{equation} 
\end{cor}

\begin{proof}
The proof of the corollary follows immediately from Theorem \ref{tail bounds}
and Lemma \ref{ctble-unctble-sups}.
\end{proof}

The next theorem shows an empirical quantile CLT holds on $[0,T]$ for the symmetric stable processes discussed here, which contrasts with the remarks at the end of section 3 showing that the empirical CLT for such processes fails. It also extends the results of \cite{swanson-scaled-median} and \cite{swanson-fluctuations} for Browning motion, showing the CLT is uniform in $\alpha \in I$, where $I$ is a closed subinterval of $(0,1)$.

\begin{thm}\label{quant-CLT-near-zero}
Let $\{X(t)\colon t\ge 0\}$ be a symmetric $r$-stable process with stationary, independent increments, cadlag sample paths, and such that $P(X(0)=0)=1$. Also, assume the  empirical quantile processes $\tau_{\alpha}^n(t)$ are built from i.i.d copies of $\{X(t): t \ge 0\}$ with cadlag paths, and $I$ is a closed subinterval of $(0,1)$. Then, the quantile processes
\begin{equation}\label{quant-CLT-near-zero-1}
\{\sqrt n (\tau_{\alpha}^n(t)- \tau_{\alpha}(t)) \colon\  n \geq 1\}
\end{equation}
satisfy the CLT in $\ell_{\infty}([0,T] \times I)$ with centered Gaussian limit process 
\begin{equation}\label{quant-CLT-near-zero-2}
\left\{W(t,\alpha)\colon\  (t,\alpha) \in [0,T] \times I \right\},
\end{equation}
where $W(0,\alpha)=0, \alpha \in I$, 
\begin{equation}\label{quant-CLT-near-zero-3}
W(t,\alpha) = \frac{G(t,\tau_{\alpha}(t))}{f(t, \tau_{\alpha}(t))}, (t,\alpha) \in (0,T] \times I,
\end{equation}
and for $(s, \beta),(t,\alpha) \in (0,T] \times I$ the covariance function is given by
\begin{equation}\label{quant-CLT-near-zero-4}
\E(W(s,\beta)W(t,\alpha)) = \frac{P(X(s) \le \tau_{\beta}(s),X(t) \le \tau_{\alpha}(t))- \alpha \beta}{f(s,\tau_{\beta}(s))f(t, \tau_{\alpha}(t))}.
\end{equation}
\end{thm}

\begin{rem}\label{median-CLT-near-zero-1}
Since the process $\{X(t): t \ge 0\}$ in Theorem \ref{quant-CLT-near-zero} is scalable with parameter $\frac{1}{r}$, it is easy to check that $\tau_{\alpha}(t) = t^{\frac{1}{r}}\tau_{\alpha}(1)$ for $(t,\alpha) \in [0, \infty)\times (0,1)$. In addition, since the density of $X(t)$ is strictly positive for each $t>0$, it is easy to check for $t>0, x \in \mathbb{R}$ that 
$$
f(t,x)=t^{-\frac{1}{r}}f(1, xt^{-\frac{1}{r}}).
$$
Thus for $t>0, \alpha \in (0,1)$,
$$
f(t,\tau_{\alpha}(t))= t^{-\frac{1}{r}}f(1,\tau_{\alpha}(1)),
$$
and for $(t,\alpha), (s,\beta) \in (0,T]\times I$, (\ref{quant-CLT-near-zero-4}) becomes
\begin{equation}\label{quant-CLT-near-zero-14}
\E(W(s,\beta)W(t,\alpha)) = s^{\frac{1}{r}}t^{\frac{1}{r}}[\frac{P(X(s) \le \tau_{\beta}(s),X(t) \le \tau_{\alpha}(t))- \alpha \beta}{f(1,\tau_{\beta}(1))f(1, \tau_{\alpha}(1))}].
\end{equation}
Furthermore, since  $W(0,\alpha)=0, \alpha \in I$, we  also have (\ref{quant-CLT-near-zero-14}) when  $(t,\alpha), (s,\beta)$ $\in [0,T]\times I$.

To get the covariance for the limiting median process, set $I=\{\frac{1}{2}\}$. Then, since $\tau_{\frac{1}{2}}(t)=0$  for all $t \in [0,T]$, we have the limiting Gaussian process such that $P(W(0,\frac{1}{2})=0)=1$ and for $s,t \in [0,T]$ its covariance is
\begin{equation}\label{quant-CLT-near-zero-15}
\E(W(s,\frac{1}{2})W(t,\frac{1}{2})) =s^{\frac{1}{r}}t^{\frac{1}{r}} \frac{P(X(s) \le 0,X(t) \le 0)- \frac{1}{4}}{f(1,0)f(1,0)}.
\end{equation}
The density $f(t,x)$ is as in (\ref{ctble-unctble-1}), and hence
\begin{equation}\label{quant-CLT-near-zero-16}
f(t,0)= \frac{\int_{\mathbb{R}} \exp\{-|u|^r\}du}{2\pi(ct)^{\frac{1}{r}}}, t>0,
\end{equation}
which implies for $s,t \in [0,T]$ that
\begin{equation}\label{quant-CLT-near-zero-17}
\E(W(s,\frac{1}{2})W(t,\frac{1}{2})) = \frac{4\pi^2(c^2st)^{\frac{1}{r}}}{(\int_{\mathbb{R}} \exp\{-|u|^r\}du)^2 }[P(X(s) \le 0,X(t) \le 0)- \frac{1}{4}].
\end{equation}
In these examples the sample paths of the input process $X$  are assumed to be cadlag, and when  $r=2$ they could be assumed to be continuous as $X$ is then a Brownian motion. Hence, after the proof of Theorem \ref{quant-CLT-near-zero} we will examine the consequences of these path properties for the quantile CLT. At this time we also will discuss results comparable to those for sample continuous Brownian motion when the input data comes from any sample continuous fractional Brownian motion. 
\end{rem}

\begin{proof}

Let
$$
W_n(t,\alpha)= \sqrt n((\tau_{\alpha}^n(t)- \tau_{\alpha}(t)), t \in [0,T], \alpha \in I,n \geq 1.
$$
Then, $P(W_n(0,\alpha)=0)=$1 for $\alpha \in I, n \ge 1,$ and  the finite dimensional distributions of $W_n$ converge to the centered Gaussian distributions given by the covariance function in (\ref{quant-CLT-near-zero-4}) for $t \in(0,T], \alpha \in I$. Hence Theorem 1.5.4 and Theorem 1.5.6 of \cite{vw} combine to imply the quantile processes $\{W_n: n\ge 1\}$ satisfy the CLT in $\ell_{\infty}([0,T]\times I)$, where the limiting centered Gaussian process has  the covariance in (\ref{quant-CLT-near-zero-4}), provided for every $\epsilon >0, \eta>0 $ there is a partition
\begin{equation}\label{quant-CLT-near-zero-5}
[0,T] \times I= \cup_{i=1}^k E_i
\end{equation}
such that
\begin{equation}\label{quant-CLT-near-zero-6}
\limsup_{n \rightarrow \infty}P^{*}( \sup _{1 \le i \le k} \sup_{(t,\alpha),(s,\beta) \in E_i}|W_n(t,\alpha) - W_n(s,\beta)| > \epsilon) \le \eta.
\end{equation}

Since $I$ is a closed subinterval of $(0,1)$, there is an $\alpha^{*} \in (\frac{1}{2}, 1)$ such that $I \subseteq A= [1-\alpha^{*}, \alpha^{*}]$. For $\delta>0$ and $E_1=[0,\delta] \times I$ observe that
\begin{align*}
&P^{*}( \sup_{(t,\alpha),(s,\beta) \in E_1}|W_n(t,\alpha) - W_n(s,\beta)| >\frac{ \eta \wedge \epsilon}{2})\\
 \le~ &2P^{*}(\sup_{ s \in [0,\delta], \alpha \in I} |W_n(s,\alpha)| >  \frac{ \eta \wedge \epsilon}{4}).
\end{align*}
Hence, (\ref{ctble-unctble-0}) and (\ref{cont-zero}) imply there is a $\delta=\delta(\frac{\eta \wedge \epsilon}{4})$ such that 
\begin{equation}\label{quant-CLT-near-zero-7}
\limsup_{n \rightarrow \infty} P^{*}( \sup_{(s,\alpha),(t,\beta) \in E_1}|W_n(s,\alpha) - W_n(t,\beta)| >\frac{ \eta \wedge \epsilon}{2}) \le 2(\frac{ \eta \wedge \epsilon}{4})\le \frac{\eta}{2}.
\end{equation}

Now Theorem \ref{quantilethm} above implies the CLT for $\{W_n(t,\alpha): (t, \alpha) \in [\delta,T]\times I\} $ in $\ell_{\infty}([\delta,T] \times I)$, and hence Theorem 1.5.4 of \cite{vw} implies that there is a partition  $[\delta, T]\times I= \cup_{i=2}^k E_i$ such that 
\begin{equation}\label{quant-CLT-near-zero-8}
\limsup_{n \rightarrow \infty}P^{*}( \sup _{2 \le i \le k} \sup_{(s,\alpha),(t,\beta) \in E_i}|W_n(s,\alpha) - W_n(t,\beta)| >\frac{ \eta \wedge \epsilon }{2}) \le \frac{ \eta \wedge \epsilon}{2}.
\end{equation}
Combining (\ref{quant-CLT-near-zero-7}) and (\ref{quant-CLT-near-zero-8}) we have (\ref{quant-CLT-near-zero-5}) and (\ref{quant-CLT-near-zero-6}), and hence the theorem is proved.
\end{proof}

Now we turn to the question of how special sample path properties of the input process $\{X_t: t \in E\}$ influence our quantile CLT's. To be more specific, recall that the CLT results we have established for empirical quantile processes, hold uniformly in the space $\ell_{\infty}(E\times I)$, and the limiting Gaussian process $\{W(t,\alpha): (t,\alpha) \in E\times I\}$, almost surely,  has a version with paths which are bounded and uniformly continuous with respect to its own $L_2$ distance $d_W$ on $E\times I$. In particular, this guarantees that the measure induced by the Gaussian process on $\ell_{\infty}(E\times I)$ is supported on the subspace $C_{L_2}(E\times I)$ of $\ell_{\infty}(E\times I)$, where the subscript $L_2$ is written to indicate the topology on $E \times I$ is that given by the Gaussian process $L_2$ distance. Hence, if the input process $\{X_t: t \in E\}$ is assumed sample continuous on $(E,e_1)$, where $e_1$ is a metric on $E$, when does our quantile CLT
  with $\alpha \in (0,1)$ fixed hold on the space of $e_1$ continuous paths? If $E=[0,T]$ with metric the usual Euclidean distance,  and the input process has cadlag sample paths on $[0,T]$, a similar question can be asked if the quantile CLT holds in some related space of functions. Since processes with continuous paths or cadlag paths are typical of many examples throughout probability and statistics, these are natural questions,  but they also relate to some recent results of Jason Swanson. That is, he established a CLT in the space of continuous functions on $[0,T]$ for the median process obtained from sample continuous Brownian motions in \cite{swanson-scaled-median}, and for other individual quantile levels $\alpha \in (0,1)$ in \cite{swanson-fluctuations}. These results will follow from our next theorem, and are established in a remark following its proof.

Since the empirical quantile processes have jumps as $\alpha$ ranges over $(0,1)$, to state our theorem providing some facts related to these questions, we need the following function spaces. If $e_1$ is a metric on $E$ we set
\begin{equation}\label{quant-CLT-functions-1}
\mathbb{C}_{e_1}(E)=\{z: z \rm {~is~continuous~on~} (E,e_1)\},
\end{equation}
if $E=[0,T]$ we assume $e_1$ is the usual Euclidean distance and let
\begin{equation}\label{quant-CLT-functions-2}
\mathbb{D}_1([0,T])=\{z: z \rm {~is~cadlag~on~} [0,T]\},
\end{equation}
where right and left limits are taken with respect to $e_1$ on $[0,T]$, and for $I=[a,b]$ a closed subinterval of $(0,1)$ we set
\begin{equation}\label{quant-CLT-functions-3}
\mathbb{D}_2(I)=\{z: z \rm {~is~left~continuous~on~} (a,b], \rm {~and~ has~right~limits~on~}[a,b)\},
\end{equation}
where right and left limits are taken with respect to the usual Euclidean distance $e_2$ on $I$. We also define the closed subspaces of $\ell_{\infty}(E\times I)$ given by
\begin{equation}\label{quant-CLT-functions-4}
\mathbb{C}_{e_1}(E) \otimes\mathbb{D}_2(I)=\{ f(\cdot,\alpha) \in \mathbb{C}_{e_1}(E) ~\forall \alpha \in I  \text{ and } f(t,\cdot) \in \mathbb{D}_2(I)~\forall t\in [0,T]\},
\end{equation}
and 
\begin{align}\label{quant-CLT-functions-5}
&\mathbb{D}_1([0,T])\otimes\mathbb{D}_2(I)=\{ f(\cdot,\alpha) \in \mathbb{D}_1([0,T]) ~\forall \alpha \in I\\
&\quad  \text{ and } f(t,\cdot) \in \mathbb{D}_2(I)~\forall t\in [0,T]\}.\notag
\end{align}
Both $\mathbb{C}_{e_1}(E) \otimes\mathbb{D}_2(I)$ and $\mathbb{D}_1([0,T])\otimes\mathbb{D}_2(I)$ are closed subspaces of\break $\ell_{\infty}(E \times I))$.

\begin{thm}\label{quant-CLT-paths}
Let $\{X_t: t \in E\}$ be the input process for the empirical quantile processes  defined for $t \in E, \alpha \in (0,1), n \ge 1,$ by
$$
W_n(t,\alpha) = \sqrt n(\tau_{\alpha}^n(t) - \tau_{\alpha}(t)),
$$
and assume they satisfy the empirical quantile CLT in $\ell_{\infty}(E \times I)$ with Gaussian limit $\{W(t,\alpha): (t,\alpha) \in E\times I\}$. Let $d_W$ denotes the $L_2$ distance of $W$ on $E \times I$, and assume the identity map $j$ on $E\times I$ is continuous from the $e_1 \times e_2$ topology on $E\times I$ to the $d_W$ topology, and that $\tau_{\alpha}(\cdot) \in \mathbb{C}_{e_1}(E)$ for every $ \alpha \in I$. Then, we have:

(i) If $\{X_t: t \in E\}$ has version with paths in $\mathbb{C}_{e_1}(E)$, then the empirical quantile CLT holds in the Banach subspace $\mathbb{C}_{e_1}(E) \otimes\mathbb{D}_2(I)$ of $\ell_{\infty}(E\times I)$. In particular, if $\alpha \in (0,1)$ is fixed, then the CLT will hold in the space of continuous functions $\mathbb{C}_{e_1}(E)$ with the topology that given by the sup-norm. 

(ii) If $E=[0,T]$ and $\{X_t: t \in E\}$ has version with paths in $\mathbb{D}_{1}([0,T])$, we have the empirical quantile CLT holding in the Banach subspace $\mathbb{D}_{1}([0,T]) \otimes\mathbb{D}_2(I)$ of  $\ell_{\infty}([0,T]\times I)$. Hence, if $\alpha \in (0,1)$ is fixed, then the CLT will hold in the space of functions $\mathbb{D}_{1}([0,T])$ with the topology that given by the sup-norm.

\end{thm} 

\begin{proof}
If $\{X_t: t \in E\}$ has a version with paths in $\mathbb{C}_{e_1}(E)$, then taking i.i.d. copies of this continuous version to build the quantile process, the proof of Lemma 5.3 implies one has with probability one that $W_n(t,\alpha)$ is continuous on $( E,e_1)$ for each $\alpha \in (0,1)$. In addition, for each $n \ge 1$, by Lemma 5.3 we have $\alpha \rightarrow \tau_{\alpha}^n(t)$ is in $ \mathbb{D}_2(I)$ for all $t\in E$, and therefore 
\begin{equation}\label{quant-CLT-functions-6}
P(W_n(\cdot,\cdot) \in \mathbb{C}_{e_1}(E) \otimes\mathbb{D}_2(I))=1.
\end{equation}
Moreover, since we are assuming the identity map $j$ is continuous from the $e_1 \times e_2$ topology to the $d_W$ topology on $E\times I$, with $W$ having a version with paths in $C_{L_2}(E \times I)$, it follows from the fact that the composition of continuous maps is continuous that 
$W$ has a version such that
\begin{equation}\label{quant-CLT-functions-7}
P(W(\cdot, \cdot) \in \mathbb{C}_{e_1}(E) \otimes\mathbb{D}_2(I))=1.
\end{equation}
Combining (\ref{quant-CLT-functions-6}) and (\ref{quant-CLT-functions-7}), an easy application of the portmanteau theorem implies (see Theorem 1.3.10 of \cite{vw} for the details) the CLT will hold on $\mathbb{C}_{e_1}(E) \otimes\mathbb{D}_2(I)$ with the topology that is given by the sup-norm. Thus (i) holds.

The proof of (ii) is entirely similar, since the assumptions of (ii) and Lemma \ref{ctble-unctble-sups} imply that (\ref{quant-CLT-functions-6}) holds with $\mathbb{C}_{e_1}(E) \otimes\mathbb{D}_2(I)$ replaced by $\mathbb{D}_1([0,T])\otimes\mathbb{D}_2(I)$. Moreover, since we always have 
$$
\mathbb{C}_{e_1}([0,T]) \otimes\mathbb{D}_2(I)\subseteq \mathbb{D}_1([0,T])\otimes\mathbb{D}_2(I),
$$
and the argument for (\ref{quant-CLT-functions-7}) is valid under (ii), we have that (\ref{quant-CLT-functions-6}) holds with this replacement.
Hence (ii) is verified as before.
\end{proof}

\begin{rem}\label{median-CLT-near-zero-2}

In this remark we provide specific applications of Theorem \ref{quant-CLT-paths}. Our first application assumes the input process $X$ is a cadlag symmetric r-stable process with stationary independent increments on $[0,T]$ with $P(X(0)=1$, and shows that under these conditions the empirical quantile CLT holds in the Banach space $\mathbb{D}_1([0,T])$ with the sup-norm. The special case r=2 implies $X$ is Brownian motion, and if the quantile CLT is built from i.i.d sample continuous Brownian motions, then we will also see that for fixed $\alpha \in (0,1)$ the empirical quantile CLT holds in the Banach space $\mathbb{C}_{e_1}([0,T])$ with the sup-norm. As we mentioned earlier, this implies the quantile CLT for medians in  \cite{swanson-scaled-median}, and for other individual quantile levels $\alpha \in (0,1)$ in \cite{swanson-fluctuations}. A major step in these results will be the use of Theorem \ref{quant-CLT-near-zero} to establish the empirical CLT. 

In the second application the input  process $X$ is a sample continuous fractional Brownian motion, and here for fixed $\alpha \in (0,1)$ we again have the empirical quantile CLT in $\mathbb{C}_{e_1}([0,T])$ with sup-norm. However, for this class of examples we will only outline the necessary arguments as they are much the same as those for the stable processes. Hence we now turn to that case.

First we observe that if $X$ is a cadlag symmetric r-stable process with stationary independent increments on $[0,\infty]$ with $P(X(0)=1$, then for $T>0$ fixed and $I$ a closed subinterval of $(0,1)$ we have the empirical quantile CLT of Theorem \ref{quant-CLT-near-zero}.
Furthermore, since $\{X_{ct}\colon\  t  \ge 0\}$ is equal in distribution to  $c^{\frac{1}{r}}\{X_t\colon\  t  \ge 0\}$ for $c>0$, it easily follows that  $\tau_{\alpha}(t)= t^{\frac{1}{r}} \tau_{\alpha}(1)$ is jointly continuous in $(t,\alpha) \in [0,\infty) \times (0,1)$. Hence Theorem \ref{quant-CLT-paths} implies the claims made above for the stable process inputs provided we show the identity map $j$ on $[0,T]\times I$ is continuous from the Euclidean topology to the $d_W$ topology on $[0,T] \times I$, where (\ref{quant-CLT-near-zero-14}) and  $(s,\beta), (t,\alpha) \in [0,T]\times I$ implies the $L_2$ distance
$d_W$ is given by
\begin{align*}
d_W^2((s,\beta),(t,\alpha)) &= \frac{s^{\frac{2}{r}}(\beta -\beta^2)}{f^2(1,\tau_{\beta}(1))} + \frac{t^{\frac{2}{r}} (\alpha -\alpha^2)}{f^2(1, \tau_{\alpha}(1))}\\
&\quad -
2s^{\frac{1}{r}}t^{\frac{1}{r}}  \frac{[P(X_s \le \tau_{\beta}(s),X_t \le \tau_{\alpha}(t)) - \alpha\beta]}{f(1,\tau_{\beta}(1))f(1, \tau_{\alpha}(1))}.
\end{align*}

Hence, fix $(t,\alpha) \in [0,T] \times I$, and assume $e((s,\beta),(t,\alpha))=e_1(s,t)+e_2(\alpha,\beta) \rightarrow 0$. Then, if $t=0$ it is obvious that the identity map $j$ is continuous at $(0,\alpha), \alpha \in I$, as asserted, i.e.  $e((s,\beta),(0,\alpha)) \rightarrow 0$ implies  $d_W((s,\beta),(0,\alpha)) \rightarrow 0$ since the density $f(1,x)$ is strictly positive and continuous in $x \in \mathbb{R}$ and $\tau_{\beta}(1)\rightarrow \tau_{\alpha}(1)$ as $\beta \rightarrow \alpha>0$. Moreover, for $t \in (0,T]$ and $\alpha \in I$ fixed, the continuity is again obvious provided we show  $e((s,\beta),(t,\alpha)) \rightarrow 
 0$ implies
\begin{equation}\label{quant-CLT-near-zero-41}
P(X_s \le \tau_{\beta}(s),X_t \le \tau_{\alpha}(t)) \rightarrow \alpha.
\end{equation}

To verify (\ref{quant-CLT-near-zero-41}) we observe that if $A_n$ and $B_n$ are sequences of random variables on the same probability space with $\rm {plim}_{n \rightarrow \infty} A_n=A$ and $\rm {plim}_{n \rightarrow \infty}B_n=B$, then the random vector $(A_n,B_n)$ converges to $(A,B)$ in probability, and $P(A_n \le x,B_n \le y)$ converges uniformly to $P(A\le x,B \le y)$ for $(x,y) \in \mathbb{R}^2$ provided $(A,B)$ has a continuous distribution function. In particular, if $A=B$, then $(A_n,B_n)$ converges in probability to $(A,A)$,
and  $P(A_n \le x,B_n \le y)$ converges uniformly to $P(A\le x,B \le y)=P(A \le x \wedge y)$ for $(x,y) \in \mathbb{R}^2$ provided $A$ has a continuous distribution function. Now $X_t$ is continuous in probability on $[0,T]$, and $\tau_{\beta}(s)= s^{\frac{1}{r}}\tau_{\beta}(1)$ is jointly continuous in $s \in [0,T]$ and $\beta \in (0,1)$, so setting $A_n=A=X_t$ with $ t>0, \alpha \in I$ fixed, and $B_n=X_{s_n}, s_n  \rightarrow t, \beta_n \rightarrow \alpha, B=X_t$, we then have
$$
\lim_{s_n \rightarrow t} P(X_{s_n} \le \tau_{\beta_n}(s_n), X_t \le \tau_{\alpha}(t))=P( X_t \le \tau_{\alpha}(t))=\alpha.
$$
Since the sequences $s_n$ and $\beta_n$ with the stated properties are arbitrary, we have  (\ref{quant-CLT-near-zero-41}). Thus the claims regarding the stable processes of this remark are established.

Now we turn to the application of Theorem \ref{quant-CLT-paths} to fractional Brownian motions. Hence let  $\{X_t\colon\  t  \ge 0\}$ be a centered sample continuous $\gamma$-fractional Brownian motion for $0<\gamma< 1$ such that $X_0=0$ with probability one and $\E(X_t^2)= t^{2\gamma}$ for $t \ge 0$. 
Take $I$ a closed subinterval of $(0,1)$, and assume the empirical quantile processes built from i.i.d. copies of $X$ with continuous paths are given for $t \ge 0, \alpha \in (0,1), n \ge 1,$ as in (\ref{quant-CLT-near-zero-1}). Since $\{X_{ct}\colon\  t  \ge 0\}$ is equal in distribution to  $c^{\gamma}\{X_t\colon\  t  \ge 0\}$ for $c>0$, it follows that  $\tau_{\alpha}(t)= t^{\gamma} \tau_{\alpha}(1)$ is jointly continuous in $(t,\alpha) \in [0,\infty) \times (0,1)$. Moreover, the analogue of the argument given above for stable processes implies the identity map $j$ on $[0,T] \times I$ is continuous from the Euclidean topology to the $L_2$ distance $d_W$ on $[0,T] \times I$, where in this situation 
\begin{equation}\label{quant-CLT-near-zero-47}
d_W^2((s,\beta),(t,\alpha)) = 2\pi \bigl\{s^{2\gamma}  \exp\{\tau_{\beta}^2(1)\}(\beta -\beta^2) + t^{2\gamma} \exp\{\tau_{\alpha}^2(1)\}(\alpha -\alpha^2)-~~~~~~~~~
\end{equation}
$$
~~~~~~~~~~~~~~~~~~~~~~~~2s^{\gamma}t^{\gamma}  \exp\{\frac{1}{2}(\tau_{\alpha}^2(1) + \tau_{\beta}^2(1)) \}[P(X_s \le \tau_{\beta}(s),X_t \le \tau_{\alpha}(t)) - \alpha\beta]\bigr\}
$$
for $(s,\beta), (t,\alpha) \in [0,T]\times I$. Hence we will have empirical CLT results as in part (i) of Theorem \ref{quant-CLT-paths} with $E=[0,T]$, provided we can show we that Theorem \ref{quant-CLT-near-zero}
applies when the input process $X$ is a fractional Brownian motion. 

Except for $\gamma=\frac{1}{2}$, $X$ does not have independent increments, 
but the results of Theorem \ref{tail bounds} still hold in this setting. Moreover, from the proof of Lemma \ref{ctble-unctble-sups} the empirical processes are such that
\begin{equation}\label{quant-CLT-functions-50}
P(W_n(\cdot,\cdot) \in \mathbb{C}_{e_1}(E) \otimes\mathbb{D}_2(I))=1.
\end{equation}
Furthermore, Corollary \ref{prob-cont-zero} and Theorem \ref{quant-CLT-near-zero} hold when the input processes are sample continuous fractional Brownian motions, and hence the proof of Theorem \ref{quant-CLT-near-zero} implies the
empirical quantile CLT in $\ell_{\infty}([0,T]\times I)$ in this setting. The limiting Gaussian process $W(t,\alpha)$ has covariance as in (\ref{quant-CLT-near-zero-4}) for $s,t \in (0,T]$ and zero for $s$ or $t$ equal to zero, and hence the covariance for $W$
is as given in (\ref{quant-CLT-near-zero-47}). Hence the empirical quantile CLT results in part (i) of Theorem \ref{quant-CLT-paths} hold as indicated with $E=[0,T]$ and a fractional Brownian motion as the input process.

In particular, if $\alpha \in (0,1)$ is fixed, then the CLT would hold in the space of continuous functions on $[0,T]$ with the topology that given by the sup-norm. The limiting Gaussian process then has covariance as in (\ref{quant-CLT-near-zero-47}) with $\beta=\alpha$. Hence if $\alpha=\beta=\frac{1}{2}$ and $s,t \in (0,T]$, the covariance of the limiting Gaussian process is
$$
\E(W(s,\frac{1}{2})W(t,\frac{1}{2})) = \frac{P(X(s) \le 0,X(t) \le 0)- \frac{1}{4}}{f(s,0)f(t,0)}= s^{\gamma} t^{\gamma} \sin^{-1}(\frac{\E(X_sX_t)}{s^{\gamma}t^{\gamma}})
$$
where $2\E(X_sX_t)= s ^{2\gamma} + t^{2\gamma} - |s -t|^{2\gamma}$, and the second equality follows from a standard Gaussian identity. When $\gamma=1/2$ and the input process is standard Brownian motion, this gives the covariance in \cite{swanson-scaled-median}.
Thus the claims of this remark are established.
\end{rem}

\begin{rem}
The Brownian sheet also has a scaling property, and continuous paths. Hence by using the methods of the previous remark, results of the type discussed there should also hold for the sheet. We have checked these results when the input process is the 2-parameter sheet on $[0,T]\times[0,T]$, but the details differ very little from what is done in the previous remark, so they are not included.
\end{rem}

\bibliographystyle{amsalpha} 

\bibliography{mybib5-08-11,othbib8-23-11}
\end{document}